\documentclass[11pt]{amsart}

\usepackage[margin=1in]{geometry}
\usepackage[metapost]{mfpic}
\usepackage{indentfirst,latexsym,amsmath,amssymb,amsthm}
\usepackage{enumerate}
\usepackage{mathptmx}
\usepackage{amsthm}
\usepackage{cleveref}

\usepackage{t1enc}
\usepackage[cp1250]{inputenc}

\usepackage{mathrsfs}
\usepackage{calc}
\usepackage[all]{xy}
\usepackage{color}
\usepackage{colortbl}

\usepackage{mathrsfs} 

\DeclareMathAlphabet{\mathpzc}{OT1}{pzc}{L}{it} 

\def\cal#1{{\mathcal #1}} 

\usepackage{amsfonts}

  \newtheorem{Th}{Theorem}[section]
  
  \newtheorem{Prop}[Th]{Proposition}
   \newtheorem{Ass}[Th]{Assumption}
  
  \newtheorem{Lem}[Th]{Lemma}
  
  \newtheorem{Cor}[Th]{Corollary}

  \newtheorem{Rem}[Th]{Remark}

  \newtheorem{Ex}[Th]{Example}

\def\e{\varepsilon}

\def\rarr{\to}
\def\N{\mathbb{N}}
\def\ov{\overline}

\def\ANR{\mathrm{ANR}}

\def\d{\,\mathrm{d}}
\def\dist{\mathrm{dist}}
\def\Fix{\mathrm{Fix}}
\def\ind{\mathrm{ind}}

\def\deg{\mathrm{deg}}

\def\la{\langle}
\def\ra{\rangle}
\def\id{\mathrm{id}}
\def\supp{\mathrm{supp}\,}

\def\cl{\mathrm{cl}\,}

\def\vp{\varphi}

\def\R{\mathbb{R}}
\def\Q{\mathbb{Q}}

\def\Gr{\mathrm{Gr}}

\def\l{\lambda}

\def\cal#1{{\mathcal #1}}
\def\part{\partial}
\def\Coin{\mathrm{Coin}}

\newcommand{\splus}{\sigma_+}
  \newcommand{\multi}{\multimap}

  \newcommand{\eps}{\varepsilon}
  \newcommand{\be}{\begin{equation}}
  \newcommand{\ee}{\end{equation}}

  \newcommand{\bea}{\begin{eqnarray}}
  \newcommand{\eea}{\end{eqnarray}}
  \newcommand{\beab}{\begin{eqnarray*}}
  \newcommand{\eeab}{\end{eqnarray*}}

\renewcommand{\leq}{\leqslant}
\renewcommand{\geq}{\geqslant}

\author[W. Kryszewski]{Wojciech Kryszewski}
\address{Institute of Mathematics, Lodz University of Technology, Lodz, Poland}
\email{wojciech.kryszewski@p.lodz.pl}
\author[M. Maciejewski]{Mateusz Maciejewski}
\address{Faculty of Mathematics and Computer Science. Nicolaus Copernicus university in Toru\'n, Poland}
\email{mateusz.maciejewski@mat.umk.pl}

\title[Degree for weakly upper semicontinuous perturbations]{Degree for weakly upper semicontinuous perturbations of quasi-$m$-accretive operators}

\date{\today}

\subjclass[2010]{}
\numberwithin{equation}{section}
\begin{document}
\pagestyle{myheadings}
\baselineskip15pt

\begin{abstract}
In the paper we provide the construction of a coincidence degree being a  homotopy invariant detecting the existence of solutions of equations or inclusions of
the form $Ax\in F(x)$, $x\in U$, where $A\colon D(A)\multimap
 E$ is an $m$-accretive operator in a Banach space $ E$,
$F\colon K\multimap  E$ is a weakly upper semicontinuous set-valued map constrained to
an open subset $U$ of a closed set $K\subset E$.
Two different approaches will be presented. The theory is applied
to show the existence of nontrivial positive solutions of some nonlinear
second order partial differential equations with discontinuities.
\end{abstract}

\maketitle

\section{Introduction}
We study the existence of solution to the problem 
\begin{equation}\label{problem1}0\in -Au+F(u),\;u\in U,\end{equation}
where $A\colon D(A)\multi E$ is a {\em quasi-$m$-accretive operator} in a (real) Banach space $ E$, $F\colon K\multi  E$ is a {\em set-valued map}, subject to constraints  in an open set $U\subset K\subset E$ of the closed set $K$ of state constraints. Such problems appear as an abstract setting of steady states of nonlinear evolution processes of the reaction-diffusion type, where $A$ corresponds to the (nonlinear) diffusion,  while a (possibly set-valued) $F$ represents a reaction term. It is reasonable to speak of solution  to \eqref{problem1} as of {\em $A$-equilibria of $F$}. As an example we will discuss a boundary-value problem for a coupled PDE system 
\begin{equation}\label{reac-diff}
\left\{\begin{array}{l} -\Delta(\rho\circ u)(x)\in\vp\big(u(x)\big),\; u\in\R^M,\;x\in\Omega\\
u_i(x)\geq 0,\;x\in\Omega,\;i=1,...,M,\\
u|_{\partial\Omega}=0,\end{array}\right.
\end{equation}
where $\Omega\subset\R^N$ is open, the unknown  $u=(u_1,...,u_M)\in\R^M$,  $\Delta$ denotes the vectorial Laplace
operator, $\rho\colon \R^M_+\rarr\R^M_+$ and $\vp\colon \R^M_+\to\R^M$ is a set-valued perturbation.

\indent If $A\equiv 0$ (resp. $A=I$ is the identity) and $K= E$, i.e., in  the absence of constraints, problems like \eqref{problem1}, concerning  the existence of the {\em equilibria}, i.e., zeros (resp. {\em fixed points}) of set-valued maps are major topics of research and were studied by numerous authors for a long time. Classical invariants of Brouwer, Leray-Schauder and their relatives i.e.,  topological degrees or the fixed point indices, are successfully applied to get equilibria or fixed points and study their behavior. Problems with constraints
(i.e., when $K\neq E$) and their homotopy counterparts are not that well-recognized, especially when the constraining set $K$ is neither convex, smooth nor even has nonempty interior with $F$ taking values {\em outside} $K$. i.e., $F$ is {\em not} a self-map of $K$. It is to observe that the very existence results are related to some classical issues in analysis (see a survey \cite{Mawhin, Krysz},  and, e.g., \cite{KS}). In \cite{b:C-K} the existence of equilibria of a weakly tangent set-valued $F\colon U\multi E$ on locally compact {\em $\cal L$-retracts} $K$ has been addressed. Simple examples show that neither the tangency assumption nor the structural conditions concerning the constraint set $K$ can be omitted. The advantage is that the class of $\cal L$-retracts is broad. On the other hand, results from \cite{b:C-K}, being of finite dimensional nature (the local compactness of the domain) are not sufficient for applications.\\
\indent The general problem \eqref{problem1} without constraints was studied by many different authors (see, e.g., \cite{Kartsatos, CC, Asfaw}), but the methods fail in the presence of constraints. The constrained problem involving a resolvent compact $m$-accretive operator $A$  and a {\em continuous single-valued} $F$ was studied in \cite{b:C-K2}, and a degree, having its roots in \cite{Chen}, detecting coincidences of $A$ and $F$ was introduced.  An attitude for usc, i.e., {\em upper semicontinuous} set-valued $F$ with {\em compact} convex values, started in \cite{b:BEM-K}, via single-valued  approximations  leads to a strict generalization of results from \cite{b:C-K2}. Here, however, we come across another difficulty since  when studying PDI  \eqref{reac-diff}, even if $\vp$ is  usc with compact convex values, the Nemytskii operator associated to $\vp$ is only weakly usc  and never has compact values.\\
\indent This motivates us to consider \eqref{problem1} with a resolvent compact quasi-$m$-accretive $A\colon D(A)\to E$ (see a comment in Remark \ref{uwzal} (iv)), an $\cal L$-retract $K$, an open (in $K$) $U\subset K$ and a weakly usc map $F\colon U\multi E$ with convex weakly compact values (we shall see that such assumptions are optimal for the viewpoint of applications). We look for $A$-equilibria of $F$, i.e., vectors $u\in U\cap D(A)$ such that $0\in -Au+F(u)$ (equivalently $Au\cap F(u)\neq\emptyset$). Let us briefly explain the idea hidden behind our approach. In case without constraints it has been started in \cite{Chen}: eq. \eqref{problem1} is  replaced by an  equivalent  fixed-point  problem
\begin{equation}\label{problem2}u\in (I+\lambda A)^{-1}\big(u+\lambda F(u)\big), \;u\in U,\;\l>0,\end{equation}
where $(I+\l A)^{-1}$ is the {\em resolvent} of $A$. In order  to deal appropriately with constraints, we suppose that $K$ is {\em invariant} with respect to resolvents of $A$, we employ the so-called {\em weak tangency} of $F$ and an approximation approach combined  with some properties of $\cal L$-retracts.\\
\indent The paper is organized as follows. After this introduction we provide preliminaries concerning terminology, notations and auxiliary results. The second section
provides the construction of a {\em constrained topological degree}; the third section is devoted to the discussion of \eqref{reac-diff} and its relatives. It seems that even though the construction is quite long and tedious, the degree we present is a convenient tool for various problems in the field of differential equations and/or inclusions and variational inequalities involving set-valued maps that arise in the evolution processes.

\section{Preliminaries}

\noindent {\sc Notation:} The notation used throughout the paper is standard. In particular $\la x,y\ra$ is the scalar product of $x,y\in\R^N$ and $|x|=\sqrt{\la x,x\ra}$ stands for the norm of $x$. The use of function spaces ($L^p$, Sobolev etc.), linear (unbounded in general) operators in Banach spaces, $C_0$ semigroups is standard. Given a metric space $(X,d)$, $K\subset X$ and $x\in X$, $d_K(x)=\dist(x,K):=\inf_{y\in K}d(x,y)$, by $\ov K$  and $\part K$ we denote the closure and the boundary of $A$; if $x\in K$, we write $y\stackrel{K}{\rarr} x$ if $y\rarr x$ and $y\in K$.\\
\indent  $( E,\|\cdot\|)$ is a real Banach space, $B:=\{x\in E\mid \|x\|<1\}$ (resp $\ov B$) is the unit open (resp. closed) ball in $ E$; $ E^*$  stands for the  dual of $ E$; if $x\in  E$, $p\in E^*$, then $\la x,p\ra:=p(x)$; by default $ E^*$ is normed. $ E^*_w$ stands for $ E^*$ endowed with the weak topology; by $\rightharpoonup$ we denote the weak convergence. $B(u_0,r)$ \big(resp. $D(u_0,r)$\big) is the open  (resp. closed) ball around $u_0\in E$ of radius $r>0$. Note that $B(u_0,r)=u_0+rB$.

\noindent {\sc Set-valued maps:}\label{set-val}  Terminology in set-valued analysis is as in  \cite[Sec. 3.1]{A-E} or \cite[Sec. 1.3]{A-F}: $F\colon X\multi E$ is a set-valued map $F$ defined on a metric space $X$ with at least {\em closed} values in $ E$;  $\Gr(F):=\{(x,y)\in X\times E\mid y\in F(x)\}$ is the {\em graph} of $F$. A map $F\colon X\multi E$ is  {\em $H$-usc} ($H$ stands for `Hausdorff' and usc abbreviates `upper semicontinuous') if for any $x\in X$ and $\eps>0$ there is $\delta>0$ such that
$F(y)\subset F(x)+\eps B=\{z\in E\mid d\big(z,F(x)\big)<\eps\}$ if $y\in X$ and $d(y,x)<\delta$. We collect some well-known facts concerning $H$-usc maps:\\
\indent (a) An usc map is $H$-usc and an $H$-usc map with compact values is usc.\\
\indent (b)  An $H$-usc map with weakly compact values  is {\em weakly} usc, i.e., usc with respect to $ E^*_w$.\\
\indent (c) If $F$ is weakly usc, then it is uhc (i.e., {\em upper hemicontinuous}, see \cite[Sec 3.2]{A-E}). An uhc map with  convex weakly compact values is weakly usc. \\
\indent (d)  If $F$ has convex weakly compact values, then it is weakly usc if and only if given a sequences $x_n\to x$ in $X$ and $y_n\in F(x_n)$ for all $n\in\N$, there is a subsequence $y_{n_k}\rightharpoonup y\in F(x)$. In particular such maps are {\em locally bounded}, i.e., for any $x\in X$, there is $\eps>0$ such that $\sup\{\|v\|\mid v\in F(y), d(x,y)<\eps\}<\infty$.

Below we strongly rely on the following  lemmata.
\begin{Lem}\label{app1} {\em (comp. \cite[Lemma 3.2]{KS})} Let $X$ be a metric space, $\Phi\colon X\multi E$ an $H$-usc set-valued map with convex values and let $\xi\colon X\times E\to \R$ be such that for any $(x,v)$, $\xi(\cdot,v)$ is usc (as a real function) and $\xi(x,\cdot)$ is convex. If for $x\in X$, there is $v_x\in\Phi(x)$ such that $\xi(x,v_x)\leq 0$, then for any $\eps>0$ and a continuous function $\eta\colon X\to (0,\infty)$ there is a continuous $f\colon X\to E$ such that $f(x)\in\Phi\Big(B\big(x,\eta(x)\big)\Big) +\eta(x)B$ and $\xi(x,f(x))<\eps$ for $x\in X$.
\end{Lem}
\begin{proof}  For any $y\in X$ let
$$V(y):=B\left(y,\frac{\eta(y)}{2}\right)\cap \left\{z\in X\mid \Phi(z)\subset \Phi(y)+\frac{\eta(y)}{2}B,\;\eta(y)<2\eta(z)\right\}.$$
Clearly $\cal V:=\{V(y)\}_{y\in X}$ is an open cover of $X$. Let $\mathcal W$ of $X$ be an open cover of $X$ {\em star refining}  $\cal V$. For  $x\in X$ let $v_x\in \Phi(x)$ with $\xi(x,v_x)\leq 0$. For $W\in {\mathcal W}$ and $x\in W$, $T_W(x):=\big\{y\in W\mid \xi(y,v_x)<\eps\big\}$.
Clearly $x\in T_W(x)$ and $T_W(x)$ is open since $\xi(\cdot,v_x)$ is usc. Hence ${\mathcal T}:=\{T_W(x)\}_{W\in {\mathcal W},\,x\in W}$ is an open cover of $X$. Let $\{\lambda_s\}_{s\in S}$ be a  partition of unity subordinated to $\mathcal T$, i.e., for any $s\in S$, there is $W_s\in {\mathcal W}$, $x_s\in W_s$ with $\supp\lambda_s\subset T_s:=T_{W_s}(x_s)$. Let
$$f(x):=\sum_{s\in S}\lambda_s(x)v_s,\;\; u\in U,$$
where $v_s:=v_{x_s}$. Then $f\colon X\to E$ is well-defined and continuous. For $x\in X$ let $S(x):=\{s\in S\mid\lambda_s(x)\neq 0\}$. If $s\in S(x)$, then $x\in T_s$, i.e., $\xi(x,v_s)<\eps$. Since $\xi(x,\cdot)$ is convex, $\xi(x,f(x))<\eps$. For any $s\in S(x)$,
$x_s, x\in T_s\subset W_s$. Since $\mathcal W$ star refines $\mathcal V$,  for  $s\in S(x)$ points $x, x_s$ belong to the star of $x$ with respect to $\cal W$, i.e.,
$$x,x_s\in\bigcup_{\{W\in {\mathcal W}\mid x\in W\}}W\subset V(y)$$
for some $y\in X$. Thus $\|x-y\|<\frac{\eta(y)}{2}<\eta(x)$, i.e., $y\in B(x,\eta(x))$ and $v_s\in \Phi(x_s)\subset \Phi(y)+\frac{\eta(y)}{2}B$ for $s\in S(x)$. This together with the convexity of $\Phi(y)$ shows that
$$f(x)\in\Phi(y)+\eta(x)B\subset \Phi\Big(B\big(x,\eta(x)\big)\Big)+\eta(x)B.\eqno\qedhere$$
\end{proof}
\begin{Lem}\label{app2} If $\Phi\colon X\times [0,1]\multi E$ is $H$-usc, then for any $\eps>0$ there is a continuous function $\delta\colon X\to (0,1)$ such that
$$\Phi\Big(B\big(x,\delta(x)\big)\times \big[0,\delta(x)\big)\Big)\subset \Phi\big(B(x,\eps)\times\{0\}\big)+\eps B,\; x\in X.$$
\end{Lem}
\begin{proof} For $x\in X$, there is $0<\delta_x<\min\{1,\eps\}$ such that $\Phi\big(B(x,2\delta_x)\times [0,\delta_x)\big)\subset \Phi(x,0)+\eps B$. Let $\{\lambda_s\}_{s\in S}$ be a  partition of unity subordinated to the open cover $\left\{B(x,\delta_x)\right\}_{x\in X}$. Hence for any $s\in S$, there is $x_s\in X$ such that $\supp\l_s\subset B(x_s,\delta_s)$, where $\delta_s:=\delta_{x_s}$. Let
$\delta(x):=\sum_{s\in S}\l_s(x)\delta_s$ for $x\in X$. Then $\delta\colon X\to (0,1)$ is well-defined and continuous. If $x\in X$, then there is $t\in S$ with $\l_t(x)\neq 0$, i.e., $x\in B(x_t,\delta_n)$, and $\delta(x)\leq \delta_t$. Then
$\Phi\Big(B\big(x,\delta(x)\big)\times \big[0,\delta(x)\big)\Big)\subset \Phi\big(B(x_t,2\delta_t)\times [0,\delta_t)\big)\subset \Phi(x_t,0)+\eps B\subset \Phi(B(x,\eps)\times\{0\})+\eps B$.
\end{proof}

\noindent {\sc Accretive operators} (see \cite{Barbu}):\label{accr} A (possibly set-valued) operator $A\colon D(A)\multi E$, where $D(A)\subset E$ is called {\em accretive}, if for all $x,y\in D(A)$,
$u\in Ax$, $v\in Ay$ and $\l>0$, $\|x-y\|\leq\|x-y+\l(u-v)\|$.  $A$ is {\em $m$-accretive} if it is accretive and $\mathrm{Range}(I+\lambda A)= E$ for some (equivalently for all) $\lambda>0$.  $A$ is {\em $\omega$-accretive} (resp. $\omega$-$m$-accretive), $\omega\in\R$, if $\omega I+A$ is accretive (resp. $m$-accretive). The following facts are of importance.\\
\indent (a) $A\colon D(A)\multi E$ is accretive iff for any $(u,v), (u',v')\in\Gr(A)$, there is $p\in J(u-u')$, where $J\colon  E\multi E^*$ is the (normalized) duality map (see \cite[Eq. (1.1)]{Barbu}), such that
$\la v-v',p\ra\geq 0$.\\
\indent (b) By the Lumer-Phillips Theorem \cite[Theorem 3.15]{Engel} a densely defined {\em linear} $A\colon D(A)\to X$ is $\omega$-$m$-accretive if and only if $-A$ is the generator of a $C_0$ semigroup  $\{e^{-tA}\}_{t\geq 0}$ such that $\|e^{-tA}\|\leq e^{t\omega}$ for $t\geq 0$.\\
\indent (c) If $A$ is $\omega$-$m$-accretive, $\lambda>0$ and $\lambda\omega<1$, then
$J_\lambda^A:=(I+\lambda A)^{-1}: E\to D(A)$ is well-defined single-valued, $\|J^A_\lambda u-J^A_\lambda w\|\leq (1-\lambda\omega)^{-1}\|u-w\|$. The map $E\times
(0,\lambda_0)\ni(u,\l)\mapsto J_{\l}^A u$, where $\lambda_0:=\infty$ if $\omega\leq 0$ and $\lambda_0=\omega^{-1}$ otherwise, is continuous. If $u\in\ov{D(A)}$, then
$\lim_{\l\to 0^+}J^A_\l u=x$.\\
\indent (d) If $A$ is $\omega$-$m$-accretive, then the following resolvent identities are satisfied
\begin{gather}\label{resid1}J^A_{\l_1}=J^A_{\l_2}\big(\l_2\l_1^{-1}I+(1+\l_2\l_1^{-1})J^A_{\l_1}\big),\; \l_i>0,\;\l_i\omega<1,\; i=1, 2;\\
\label{resid2}J^A_\l u=J^{\omega I+A}_{\l(1-\l\omega)^{-1}}\big((1-\l\omega)^{-1}u\big),\; \l>0,\; \l\omega<1.\end{gather}

\noindent {\sc Resolvent compact accretive operators: }\label{res-com}
An $\omega$-$m$-accretive operator $A$ is said to be {\em resolvent compact} if $J_\l$ is compact for any $\l>0$ with $\l\omega<1$ (\footnote{A continuous map is {\em compact} if it maps bounded sets into compact ones.}). If $A$ is resolvent compact, $0<\l_1\leq\l_2$ and $\l_i\omega$ for $i=1,2$, then a map $ E\times [\l_1,\l_2]\ni (u,\l)\mapsto J_\l u$ is compact.\\
\indent A family $\cal A=\{A(t)\}_{t\in [0,1]}$, where $A(t)$ is an $\omega(t)$-$m$-accretive for $t\in [0,1]$, is {\em resolvent continuous} (resp. {\em resolvent compact}) if $\omega\colon [0,1]\to\R$ is continuous and for $\lambda>0$, $\lambda\sup_{t\in [0,1]}\omega(t)<1$, the map
$$ E\times [0,1]\ni (u,t)\mapsto J^{A(t)}_\lambda u$$
is continuous (resp. compact). If $\cal A$ is resolvent compact, then $A(t)$ is resovent compact for all $t$. If $A$ is $m$-accretive, then the  family
$A(t):=\omega(t)I+\mu(t)A$, where $\omega,\mu\colon [0,1]\to \R$ are continuous, $\inf\mu>0$, is resolvent continuous; it is resolvent compact if so is $A$ (see \cite[Example 2.6]{b:C-K2}). The following {\em closedness} result is of  importance.
\begin{Prop}\label{clos} Let  $\cal A=\{A(t)\}_{t\in [0,1]}$ be a resolvent continuous family of $\omega$-$m$-accretive operators.\\
\indent {\em (i)} For sequences $(u_n)$ in $ E$, $(t_n)$ in $[0,1]$ and $(v_n)$ in $ E$ such that $u_n\in D(A(t_n))$,
$v_n\in A(t_n)x_n$ for all $n\geq 1$, if $u_n\to u_0\in M$, $t_n\to t_0\in [0,1]$ and $v_n\to v_0\in E$ (resp. $ E^*$ is uniformly convex and $v_n\rightharpoonup v_0$), then $(u_0,v_0)\in\Gr(A)$, i.e.,  $u_0\in D(A(t_0))$ and $v_0\in A(t_0) u_0$. \\
\indent {\em (ii)} If $\cal A$ is resolvent compact, then for any bounded sequences $(u_n)$ in $M$, $(v_n)$ in $E$ and $(t_n)$ in $[0,1]$ such that $u_n\in D(A(t_n))$ and $v_n\in A(t_n) u_n$,  the set $\{u_n\}$ is relatively compact.
\end{Prop}
\begin{proof} The "strong" version of (i) and (ii)  follows  from \cite[Prop. 2.7]{b:C-K2}. We have only to show (i) when $v_n\rightharpoonup v_0$. Assume first that $\omega(t)=0$ for all $t$, i.e., $\cal A$ consists of $m$-accretive operators. For $n\geq 0$, let $y_n:=J^{A(t_n)}_1(u_0+v_0)$. Then $u_0+v_0=y_n+z_n$, where $z_n\in A(t_n)y_n$, for any $n\geq 0$. The resolvent continuity implies that $y_n\to y_0$ and  hence $z_n\to z_0$. Since  $ E^*$ is uniformly convex, the duality mapping $J$ is single valued and continuous. In view of (a) on page \pageref{accr}, $\la v_n-z_n,J(u_n-y_n)\ra\geq 0$ for all $n\geq 1$. Passing with $n\to\infty$ one has
$$0\leq \la v_0-z_0,J(u_0-y_0)\ra=\la y_0-u_0,J(u_0-y_0)\ra=-\|y_0-u_0\|^2.$$
Hence $u_0=y_0$ and $v_0=z_0$, i.e., $(u_0,v_0)\in\Gr(A(t_0))$.\\
\indent Suppose  $\omega(\cdot)$ is continuous. Then for $t\in [0,1]$,
$B(t):=\omega(t) I+A(t)$, $t\in [0,1]$ is $m$-accretive and the family $\{B(t)\}$ is resolvent continuous (see \eqref{resid2}). Clearly $w_n:=\omega(t_n)u_n+v_n\in B(t_n)u_n$ and $w_n\rightharpoonup w_0=\omega(t_0)u_0+v_0$. By the above argument $(u_0,w_0)\in\Gr(B)$; hence $(u_0,v_0)\in\Gr(A)$.\end{proof}
\begin{Cor}\label{corclos} If $A$ is $\omega$-$m$-accretive and $ E^*$ is uniformly convex, then $\Gr(A)$ is closed as the subset of $ E\times E_w^*$, i.e., given a sequence $(u_n,v_n)\in\Gr(A)$ if $u_n\to u_0$ and $v_n\rightharpoonup v_0$, then $(u_0,v_0)\in\Gr(A)$. If $A$ is resolvent compact, $(u_n,v_n)\in\Gr(A)$ and sequences $(u_n)$, $(v_n)$ are bounded, then $(u_n)$ has a convergent subsequence.\hfill $\square$
\end{Cor}

\noindent {\sc Tangent cones and $\cal L$-retracts:}\label{tan-con} (a) If $K\subset E$ is closed, then the {\em Clarke tangent cone} is defined
by
\begin{equation}\label{Clarke cone}
T_K(x):=\left\{u\in E\mid \limsup_{y\stackrel{K}{\rightarrow}x, h\rarr 0+}\frac{d_K(y+hu)}{h}=0\right\}.\end{equation}
Observe that
$$\limsup_{y\stackrel{K}{\rightarrow}x,
h\rarr 0+}\frac{d_K(y+hu)}{h}=d_K^\circ(x,u)$$ is the Clarke derivative of $d_K(\cdot)$ at $x$ in the direction of $u$ (see, e.g., \cite{Clarke} or \cite[Proposition 7.3.10]{A-E} for details). Hence the map $K\times E\ni (x,u)\mapsto d_K^\circ(x,u)$ is upper semicontinuous (as a real function) and for $x\in K$, $d_K^\circ(x,\cdot)$ is  convex. If $K$ is closed convex, then (see \cite[Section 4.2]{A-F})
$$T_K(x)=\ov{\bigcup_{h>0}h^{-1}(K-x)}=\left\{u\in E\mid \lim_{h\to 0^+}\frac{d_K(x+hu)}{h}=0\right\}$$
is the tangent cone in the sense of convex analysis. Some other facts and details are to be found in \cite[Section 4.1]{A-F} or \cite[Section 4.1]{A-E}. Let us only mention that if $K=D(x_0,R)$ and $\|x-x_0\|=R$, then $u\in T_K(x)$ iff $[x-x_0,u]_+\leq 0$, where the semi-inner product (see, e.g. \cite[Section 1.6]{Vrabie}),
$$[x,y]_+:=\lim_{h\to 0^+}\frac{\|x+hy\|-\|x\|}{h},\;\; x,y\in E.$$
Observe that if $E$ is a Hilbert space, then $\|x\|[x,y]_+=\la x,y\ra$ is the inner product.

\begin{Cor}\label{app3} Let $K\subset E$ be closed and $U\subset K$. If\, $\Phi\colon U\times [0,1]\multi  E$ is $H$-usc with convex values and $\Phi(u,t)\cap T_K(u)\neq\emptyset$, $u\in U$, $t\in [0,1]$, then for any $\eps>0$ there is a continuous $f\colon U\times [0,1]\to E$ such that for $u\in U$, $t\in [0,1]$,
\begin{align}\label{app4}d_K^\circ(u,f(u,t))<\eps\;\; \text{and}\;\; f(u,t)\in \Phi(B(u,\eps)\times (t-\eps,t+\eps))+\eps B,\\
\label{+app}f(u,i)\in \Phi_i(B(u,\eps)+\eps B,\; \text{where}\;\; \Phi_i:=\Phi(\cdot,i),\; i=0,1.\end{align}
\end{Cor}
\begin{proof}  By Lemma \ref{app2} we find  a continuous $\delta\colon U\to (0,1)$ such that
\begin{align*}\Phi\Big(B\big(u,\delta(u)\big)\times \big[0,\delta(u)\big)\Big)&\subset \Phi\big(B(u,\eps)\times\{0\}\big)+\eps B;\\
\Phi\Big(B\big(u,\delta(u)\big)\times \big(1-\delta(u),1]\big)\Big)&\subset \Phi\big(B(u,\eps)\times\{1\}\big)+\eps B.\end{align*}
Letting $\xi\big((u,t),v\big):=d_K^\circ(u,v)$ and $\eta(u,t):=\min\big\{\delta(u),\eps\big\}$ for $u\in U$, $t\in [0,1]$. In view of Lemma \ref{app1}, we find a continuous $f\colon U\times I\to E$ such that
$d_K^\circ \big(u,f(u,t)\big)<\eps$ and
$$f(u,t)\in \Phi\Big(B\big(u,\eta(u,t)\big)\times \big(t-\eta(u,t),t+\eta(u,t)\big)\Big)+\eta(u,t) B\subset \Phi\big(B(u,\eps)\times (t-\eps,t+\eps)\big)+\eps B.$$
In particular
\begin{align*}f(u,0)&\in \Phi\Big(B\big(u,\eta(u,0)\big)\times \big[0,\eta(u,0)\big)\Big)\subset \Phi\Big(B\big(u,\delta(u)\big)\times \big[0,\delta(u)\big)\Big)\subset \Phi_0\big(B(u,\eps)\big)+\eps B,\\
f(u,1)&\in \Phi\Big(B\big(u,\eta(u,1)\big)\times \big(1-\eta(u,0),1\big]\Big)\subset \Phi\Big(B\big(u,\delta(u)\big)\times \big(1-\delta(u),1\big]\Big)\subset \Phi_1\big(B(u,\eps\big)+\eps B.\hfill \qedhere\end{align*}
\end{proof}

\indent (b) Let $\eps>0$. In the setting of Corollary \ref{app3} we say that $f\colon U\times [0,1]\to E$ is an {\em $\eps$-tangent (graph-) approximation} of $\Phi\colon U\times [0,1]\multi E$ \big(writing $f\in a(\Phi,\eps)$\big) if condition \eqref{app4} is satisfied. Similarly we say that  $f\colon U\to E$ is an {\em $\eps$-tangent approximation} of $F\colon U\multi E$ \big(writing $f\in a(F,\eps)$\big), if for all $u\in U$,
\begin{align}\label{app5}d_K^\circ(u,f(u))<\eps\;\; \text{and}\;\; f(u)\in F\big(B(u,\eps)\big)+\eps B.
 \end{align}

\indent (c) The closed set  $K\subset E$ is an $\mathcal L$-retract if there is $\eta>0$, a constant $L\geq 1$ and  a retraction $r\colon B(K,\eta)\to K$ such that $\|x-r(x)\|\leq L\dist(x,K)$ for all $x\in B(K,\eta):=\{x\in E\mid \dist(x,K)<\eta\}$. The class of $\cal L$-retracts is  broad; among others  it contains closed convex sets, the so-called epi-Lipschitz sets, $C^2$ manifolds and others (see \cite{}). It is clear that an $\cal L$-retract $K$ belongs to the class of metric ANRs ({\em absolute neighborhood retracts}).

\noindent {\sc Fixed point index on ANRs} (see \cite{D-G}
or \cite{Granas}): \label{index}  Let $M$ be a metric ANR. A compact map $\phi\colon  U\to M$, where $U\subset M$ is open, is {\em admissible} if the fixed point set $\Fix(\phi):=\{ x\in
U\mid\phi(x)=x\}$ is compact. By an {\em admissible homotopy} we
mean a compact map $\phi\colon U\times [0,1]\to M$ such that the set $\big\{x\in U\mid \phi(x,t)=x\; \text{for some}\; t\in [0,1]\big\}$ is compact. To any admissible map $\phi\colon U\to M$, there corresponds
an integer $\ind_M(\phi,U)$ such that:
\begin{enumerate}
\item{\em (Existence)} If $\ind_M (\phi,U)\neq 0$, then
$\Fix(\phi)\neq \emptyset$.
\item {\em (Additivity)} Let  $\phi\colon U\to M$ be an
admissible map and $U_1, U_2\subset U$ be open and disjoint sets
such that $\{x\in U\,|\, \phi(x)=x \} \subset U_1\cup U_2$. Then
$\ind_M (\phi,U) = \ind_M  (\phi_{\,|{U}_1}, U_1) +
\ind_M (\phi_{\,| {U}_2}, U_2).$
\item {\em (Homotopy invariance)} If $\psi\colon U\times [0,1]\to M$ is an admissible homotopy, then $\ind_M\big(\psi (\cdot,0),U\big)=\ind_M(\psi(\cdot,1),U)$.
\item {\em (Contraction)} If $N\subset M$ is closed and $N\in ANR$, $\phi(U)\subset N$, then $\ind_M(\phi,M)=\ind_N(\psi,U\cap N)$, where $\psi\colon U\cap N\to N$ is the contraction of $\phi$.
\item{\em (Local normalization)} If $\Theta_{x_0}\colon U\to M$, where $x_0\in M\setminus \partial U$ is given by $\Theta_{x_0} (x):=x_0$ for $x\in U$, then
$$
\ind_M (\Theta_{x_0}, U) = \left\{
\begin{array}{ll}
1 & \mbox{ if } x_0\in U\\
0 & \mbox{ if } x_0\in M\setminus \cl U.
\end{array}
\right.
$$
\item {\em (Global normalization)} If $\phi \colon  M \to M$ is
compact at large (i.e. $\phi(M)$ is relatively compact), then $\phi$ is a Lefschetz map (see \cite[Theorem (7.1)]{Granas}), its {\em generalized} Lefschetz number $\Lambda(\phi)$ is defined $\ind_M(\phi,M)=\Lambda(\phi)$ (\footnote{Interesting and important computations of the index on closed convex sets are provided in \cite{Dancer, Dancer1}.}).
\end{enumerate}

\section{The constrained coincidence degree}

We now define an invariant detecting solutions to  \eqref{problem1}, where $K\subset E$ is an $\cal L$-retract. This assumption is general; in applications, however, convex closed constraining sets are usually encountered (recall that convex closed sets are $\cal L$-retracts).
\begin{Ass}\label{assad} {\em A pair $(A,F)$ is {\em admissible} in the sense that:
\vspace{-4mm}
\begin{enumerate}
\item[(1)] an $\omega$-$m$-accretive operator $A\colon D(A)\multi E$ is a resolvent compact and $K$ is resolvent-invariant, i.e., $J_\l(K)\subset K$ for  $0<\l\leq\bar\l$, where $\bar\l\omega<1$;
\item[(2)] a map $F\colon U\multi E$, where $U\subset K$ is open (in $K$), is  {\em weakly tangent}, i.e.,  $F(u)\cap T_K(u)\neq\emptyset$ for $u\in K$;
\item[(3)] the set $C=\Coin(A,F;U):=\{u\in U\mid u\in D(A), \; Au\cap F(u)\neq\emptyset\}$ is compact.
\end{enumerate}
}\end{Ass}
\begin{Ass}\label{assadhom} {\em A pair $(\{A(t)\}_{t\in [0,1]},\Phi)$ is an admissible homotopy if:
\vspace{-4mm}
\begin{enumerate}
\item[(1)] the family of $\omega(t)$-$m$-accretive operators $\{A(t)\}_{t\in [0,1]}$  is resolvent compact; $J^{A(t)}_\l(K)\subset K$ for all $0<\l\leq\bar\l$, where $\bar\l\sup_{t\in [0,1]}\omega(t)<1$;
\item[(2)] a  map $\Phi\colon U\times [0,1]\multi$ is  {\em weakly tangent}, i.e., $\Phi(u,t)\cap T_K(u)\neq\emptyset$ for $u\in U$, $t\in [0,1]$;
\item[(3)] the set $\widetilde C:=\bigcup_{t\in [0,1]}\Coin\big(A(t),\Phi(\cdot,t)\big)$ is compact.
\end{enumerate}}
\end{Ass}
\begin{Rem}\label{uwzal} {\em Let us briefly comment the assumptions.\\
\indent (i) The structural assumption that $K$ is an $\cal L$-retract is necessary in a sense. It has been thoroughly discussed in \cite{b:BEM-K}, where examples of neighborhood retracts and tangent maps without equilibria were shown. It is moreover clear that without  tangency condition (which, in fact, prescribes directions of values taken by $F$) no equilibria of $F$ exist.\\
\indent (ii) The resolvent compactness of $A$ is the only compactness condition considered in this paper. As it seems  if $E$ is a Hilbert space, $K$ is convex, $\omega<0$ and $F$ is a set-valued map being locally Lipschitz with respect to the Hausdorff distance, then the construction of an ivariant detecting $A$-equilibria is also possible.\\
\indent (iii) Recall that $A$-equilibria, i.e., solutions to \eqref{problem1},  correspond to stationary states of the evolution process governed by $A$ and perturbed by $F$. It is therefore reasonable to assume that the `unperturbed' evolution determined  by $A$ is viable in $K$, i.e., if $u\in K$, then $S_A(t)u\in K$ for $t\geq 0$, where $\{S_A(t)\}_{t\geq 0}$ is the (nonlinear) semigroup generated by $A$. The viablity condition (i.e., the semigroup invariance of $K$) is slightly weaker then resolvent invariance (due to the so-called Crandall-Liggett formula). The invariance of this type is a subject of intensive research (see, e.g., \cite{KS}).\\
\indent (iv) Let us also address the following question. By supposition $A$ is $\omega$-$m$-accretive, i.e., $B:=A+\omega I$ is $m$-accretive. It is easy to see that \eqref{problem1} is equivalent to $0\in -Bu+G(u)$, $u\in K$, where $G=F+\omega I$.
It is clear that $B$ is resolvent compact. Moreover if $K$ is convex, $\omega<0$ and $0\in K$, then $K$ is invariant with respect to resolvents of $B$ (see \eqref{resid2}) and $G$ is weakly tangent. Therefore under these circumstances one may discuss $m$-accretive operators without loss of generality. This is not, however, the case in a general situation: the assumption of quasi-$m$-accretivity can not be replaced by $m$-accretivity.
}\end{Rem}

\vspace{-4mm}

We shall provide two constructions: in the first one we assume that the admissible perturbation $F$ is $H$-usc with weakly compact convex values; in the second one weakly usc admissible perturbations will be considered. As mentioned above the ideas behind our construction rely on an approach started by \cite{Chen} (see also \cite{Guan}) in the unconstrained and single-valued situation and \cite{b:C-K2}, where the perturbation was single-valued, too. It is to be mentioned that a different, `dynamical', approach has been presented in \cite{Cwisz1} (see also \cite{Cwisz}) for single-valued locally Lipschitz perturbations.
\subsection{$H$-upper semicontinuous perturbations} In this subsection we assume {\em additionally} that
\vspace{-4mm}
\begin{enumerate}
\item[$(D_1)$] the dual space $E^*$ is uniformly convex;
\item[$(D_2)$] $(A,F)$ is admissible and $F\colon U\multi  E$ is $H$-usc with convex weakly compact values;
\item[$(D_3)$] $(\{A(t)\}_{t\in [0,1]},\Phi)$ is admissible and $\Phi\colon U\times [0,1]\multi E$ is $H$-usc with convex weakly compact values.
\end{enumerate}

\begin{Rem}\label{discussion}{\em (1) Condition \ref{assad} (3) \big(resp. \ref{assadhom} (3)\big) holds true iff $C$ and $F(C)$ (resp. $\widetilde C$ and  $\Phi(\widetilde C\times [0,1])$) are bounded. Indeed, if $C$ is compact, then  $F(C)$ is weakly compact and, hence, bounded. If $C$ and $F(C)$ are bounded, then $C$ is relatively compact since $C\subset \{J_\l(u+\l v)\mid u\in C, v\in F(u)\}$ and $J_\l$ is compact. To see  that $C$ is closed take a sequence $(u_n)$ in $C$, $u_n\to u_0$ and $v_n\in Au_n\cap F(u_n)$, $n\geq1$. Condition $(D_2)$ implies that (a subsequence) $v_{n_k}
\rightharpoonup v_0\in F(u_0)$. In view of Corollary \ref{corclos}, $v_0\in D(A)$ and $v_0\in Au_0$, i.e. $u_0\in C$.\\
\indent (2) Observe that if $F$ usc with compact values, the provisional assumption concerning the uniform convexity of $E$ is not necessary because in this case the graph $\Gr(F)$ is closed.\\
\indent (3) In view of Corollary \ref{app3} if  $F$ (resp. $\Phi$) is admissible and  $(D_2)$ (resp. $(D_3)$) holds true, then $a(F,\eps)\neq\emptyset$  (resp. $a(\Phi,\eps)\neq\emptyset$). Moreover there are $\eps$-tangent approximations of $\Phi$ satisfying condition \eqref{+app}. }
\end{Rem}

\vspace{-4mm}

\noindent {\sc The construction:}  Let us first briefly describe the idea behind. Instead of \eqref{problem1} or \eqref{problem3} we study a problem
\begin{equation}\label{problem3}u\in J_\lambda\circ r(u+\l f(u)),\; u\in U,\end{equation}
where $\l>0$ is small enough, $r$ is an $\cal L$-retraction and $f$ is a sufficiently close tangent graph-approximation of $F$. Since the map in the RHS of \eqref{problem3} is compact, the fixed point index can be used.\\
\indent We shall proceed in several steps.

\noindent {\em Step 1}: Since $C$ is compact and $F$ is locally bounded, there are an open (in K) bounded set $W\subset K$ such that
\begin{equation}\label{cons1}C\subset W\subset \ov W\subset U,\end{equation}
and $\bar\eps>0$ such that $F(\ov W)$ and $f(\ov W)$ are bounded if $f\in a(F,\eps)$, $0<\eps\leq \bar\eps$ (here $\ov W$ is the closure of $W$ in $K$).

\noindent {\em Step 2}: Take $W$ as  in Step 1. For any $M>0$, there is $0<\eps_0=\eps_0(M)\leq\bar\eps$ depending on $M$ (and $W$)
 such that if $0<\eps\leq\eps_0$ and $f^1,f^2\in a(F,\eps)$, then for  $u\in \part W\cap D(A)$, $t\in [0,1]$ and $v\in Au$
\begin{equation}\label{cons2}\|v-f(u,t)\|\geq M\eps_0\end{equation}
where $f(\cdot,t):=(1-t)f^1+t f^2$, $t\in [0,1]$.\\
\indent Suppose to the contrary that there are sequences  $\eps_n\to 0$, $0<\eps_n\leq \bar\eps$, $(u_n)$ in $\part W\cap D(A)$ and $v_n\in Au_n$ such that $\|v_n-f_n(u_n,t_n)\|<M\eps_n$, where $f_n(\cdot,t):=(1-t)f^1_n+t f_n^2$ and $f_n^1,f_n^2\in a(F,\eps_n)$, $n\geq 1$. The sequence $(v_n)$ is bounded. In view of Corollary \ref{corclos} $(u_n)$ has a convergent subsequence;  assume wlog that $u_n\to u_0\in\part W$ and that $t_n\to t_0$. Since $f^i_n\in a(F,\eps)$, we get $f^i_n(u_n)\in F(B(u_n,\eps_n))+\eps_nB$, there is $\bar u^i_n\in U$, $\|\bar u^i_n-u_n\|<\eps_n$ and $\bar v^i_n\in F(\bar u^i_n)$ such that $\|f^i_n(u_n)-\bar v^i_n\|<\eps_n$, $i=1,2$. Clearly $\bar u^i_n\to u_0$ and (after passing to subsequence) $\bar v^i_n\rightharpoonup v^i_0\in F(u_0)$ since $F$ is weakly usc with convex weakly compact values (see page \pageref{set-val}). Thus $\bar v_n:=(1-t_n)\bar v^1_n+t \bar v^2_n\rightharpoonup v_0:=(1-t_0)v_0^1+t_0v_0^2\in F(u_0)$
and $v_n\rightharpoonup v_0$, too. Therefore $v_0\in Au_n$ by Corollary \ref{corclos}, i.e., $v_0\in\Coin(A,F;U)$, a contradiction.

\noindent {\em Step 3}: Take a retraction $r\colon B(K,\eta)\to K$ with $\|x-r(x)\|\leq Ld_K(x)$ for some $L\geq 1$.  Let $f^1,f^2\in a(F,\eps)$, where $0<\eps\leq\eps_0(L)$. We claim that there is $0<\lambda_0\leq\bar\l$ \big(see assumption \ref{assad} (1)\big) depending on $W$
such that if $0<\l\leq \l_0$, then for any $u\in\ov W$, $t\in [0,1]$, $u+\l f(u,t)\subset B(K,\eta)$, where $f(\cdot,t)=(1-t)f^1+t f^2$, and the set
$$C_\l:=\{u\in\ov W\mid u=J_\l\circ r\big(u+\l f(u,t)\big)\}$$ is contained in $W$.\\
\indent Wlog we may suppose $\bar\l$ is small enough that for $0<\l\leq\bar\l$, $\lambda \|f(u,t)\|<\eta$ for $u\in\ov W$. Hence if $0<\l\leq\bar\l$, then $u+\l f(u,t)\in B(K,\eta)$ for $u\in\ov W$ and $t\in [0,1]$. Thus $C_\l$ is well-defined. Suppose to the contrary that there are sequences $\l_n\to 0$, $0<\l_n\leq\bar\l_n$, $(u_n)$ in $\part W\cap D(A)$  and $(t_n)$ in $[0,1]$ such that
\begin{equation}\label{cons3}u_n=J_{\l_n}\circ r\big(u_n+\l_nf(u_n,t_n)\big),\;n\geq 1.\end{equation}
Hence there is $v_n\in Au_n$ with $r\big(u_n+\l_nf(u_n,t_n)\big)=u_n+\l_nv_n$, $n\geq 1$. By the definition $\cal L$-retract we see that
\begin{equation}\label{cons4}\l_n\|v_n-f(u_n,t_n)\|=\|r\big(u_n+\l_nf(u_n,t_n)\big)-\big(u_n+
\l_nf(u_n,t_n)\big)\|\leq Ld_K\big(u_n+\l_nf(u_n,t_n)\big),\; n\geq 1.\end{equation}
Observe that $d_K\big(u_n+\l_n f(u_n,t_n)\big)\leq\l_n\|f(u_n,t_n)\|$ since $u_n\in K$. In view of \eqref{cons4} $\|v_n\|\leq (L+1)\|f(u_n,t_n)\|$, i.e., the sequence $(v_n)$ is bounded. This, together with the boundedness of $(u_n)$, implies that, after passing to a subsequence, $u_n\to u_0\in\part W$. We also let wlog $t_n\to t_0$. In view of \eqref{cons2} and \eqref{cons4} for all $n\geq 1$
\begin{equation}\label{cons5}L\eps_0\leq\big\|v_n-f(u_n,t_n)\big\|\leq L\frac{d_K\big(u_n+\l_nf(u_0,t_0)\big)}{\l_n}+L\|f(u_n,t_n)-f(u_0,t_0)\|.
\end{equation}
Passing with $n\to\infty$ in \eqref{cons5}, property \eqref{app5} of $f^1$ and $f^2$ yields
\begin{gather*}L\eps_0\leq\limsup_{n\to\infty}\big\|v_n-f(u_n,t_n)\big\|\leq Ld_K^\circ\big(u_0,f(u_0,t_0)\big)\leq L\Big((1-t_0)d_K^\circ\big(u_0,f^1(u_0)\big)+t_0d_K^\circ\big(u_0,f^2(u_0)\big)\Big)<L\eps, \end{gather*}
i.e., $\eps_0<\eps$: a contradiction.

\noindent {\em Step 4}: Take $0<\eps<\eps_0=\eps_0(L+1)$ and $0<\l\leq\l_0$. Let $g^\eps_\l(u)\colon \ov W\to K$ be given by
\begin{equation}\label{cons6}g^\eps_\l(u):=J_\l\circ r\big(u+\l f(u)\big),\; u\in \ov W,\end{equation}
where $f\in a(F,\eps)$. Then $g^\eps_\l$ is well-defined. Moreover $g^\eps_\l$ is compact since so is $J_\l$. In view of Step 3, the set $\Fix g^\eps_\l:=\big\{u\in\ov W\mid u=g^\eps_\l(u)\big\}$ of fixed points of $g^\eps_\l|_{\ov W}$ is contained in $W$. Therefore we are in a position to consider the fixed point index $\ind_K\big(g^\eps_\l,W\big)$. We claim that if $0<\eps_1\leq\eps_2\leq\eps_0$, $0<\l_1\leq\l_2\leq\l_0$, $f_1\in a(F,\eps_1)$, $f_2\in a(F,\eps_2)$ and $g_i(u):=J_{\l_i}\circ r(u+\l_if_i(u))$ for $u\in\ov W$, $i=1,2$, then maps $g_i\colon \ov W\to K$ are compact and
\begin{equation}\label{cons7}\ind_K(g_1,W)=\ind_K(g_2,W).\end{equation}
To see this consider a homotopy $g\colon \ov W\times [0,1]\to K$ given by
$$g(u,t)=J_{\l(t)}\circ r\big(u+\l(t)f(u,t)\big),\;u\in\ov W, t\in [0,1],$$
where $\l(t):=(1-t)\l_1+t\l_2$, $f(\cdot,t)=(1-t)f_1+tf_2$, $t\in [0,1]$. It is clear that $g$ is compact (see page \pageref{res-com}) and
$\big\{u\in\ov W\mid u=g(u,t)\;\text{for some}\; t\in [0,1]\big\}\subset W$. Therefore \eqref{cons7} follows from the homotopy invariance of the index.

\noindent {\em Step 5}: Let open (in $K$) and bounded sets $W_1$, $W_2$ be such that $W_1\subset W_2$, $C\subset W_i\subset \ov W_i\subset U$ and  $f(\ov W_i)$ be bounded if $f\in a(F,\eps)$, where $\eps>0$ is sufficiently small. Take a sufficiently small $\l>0$, too. Then arguing as in Step 4 we show that $g^\eps_\l(u)\neq u$ for $u\in \ov W_2\setminus W_1$. The localization property of the fixed point index implies that $\ind_K\big(g^\eps_\l,W_1\big)=\ind_K\big(g^\eps_\l,W_2\big)$.

\noindent {\em Step 6}: Finally: if $W$ is taken as in Step 1,  $\eps>0$ and $\l>0$ are sufficiently small, then the index $\ind_K(g^\eps_\l,W)$ does not depend on the choice of a neighborhood retraction $r$.\\
\indent To this end suppose that $r_i\colon B(K,\eta_i)\to K$ are retractions with $\|r_i(u)-u\|\leq L_i d_K(u)$ for $u\in B(K,\eta_i)$, $i=0,1$. Let $\eta:=\min\{\eta_0,\eta_1\}$ and $L:=\max\{L_0,L_1\}$. Consider
$$g_i(u):=J_\l\circ r_i\big(u+\l f(u)\big),\; u\in\ov W,\; i=0,1,$$
where $f\in a(F,\eps)$, $0<\eps\leq\eps_0\big(L(L+2)\big)$ and $0<\l\leq\l_0$. We  show that $\ind_K(g_0,W)=\ind_K(g_1,W)$ provided $\l$ is small enough. To this aim consider a map
$$g(u,t)=J_\l\circ r_0\Big(\big(1-t\big)\big(u+\l f(u)\big)+tr_1\big(u+\l f(u)\big)\Big),\; u\in\ov W,\; t\in [0,1].$$
It is easy to see that $d_K\Big(\big(1-t\big)\big(u+\l f(u)\big)+tr_1\big(u+\l f(u)\big)\Big)\leq \l(1+t)\|f(u)\|$. Hence $g$ is well-defined when $\l$ is small.
Clearly $g(\cdot,i)=g_i$ for $i=0,1$. As in Step 5, $g$ is compact. To conclude the proof by the homotopy invariance we check that if $\l$ is sufficiently small, then $u\neq g(u,t)$ for $u\in\part W$, $t\in [0,1]$. Suppose to the contrary that
there are sequences $\l_n\searrow 0$, $(u_n)$ in $\part W$ and $(t_n)\in [0,1]$ such that
$u_n=g(u_n,t_n)$ for $n\geq 1$, i.e.,
\begin{gather*}u_n+\l_nv_n=r_0\Big(\big(1-t_n\big)\big(u_n+\l_n f(u_n)\big)+t_nr_1\big(u_n+\l_n f(u_n)\big)\Big)\\=
r_0\bigg(u_n+\l_n f(u_n)+t\Big(r_1\big(u_n+\l_nf(u_n)\big)-\big(u_n+\l_nf(u_n)\big)\Big)\bigg),\end{gather*}
where $v_n\in Au_n$, $n\geq 1$. By a straightforward computation $\l_n\|v_n-f(u_n)\|\leq L(L+2)d_K\big(u_n+\l_nf(u_n)\big)$ and hence
$$L(L+2)\eps_0\leq\limsup_{n\to\infty}\big\|v_n-f(u_n)\big\|\leq L(L+2)d_K^\circ\big(u_0,f(u_0)\big)<L(L+2)\eps.$$
A contradiction shows the assertion.

\noindent {\em Conclusion}: We define the {\em constrained degree of coincidence} of the pair $(A,F)$ by the following
formula
\begin{equation}\label{def-deg}\deg_K(A,F;U):=\lim_{\l\searrow 0}\ind_K(g^\eps_\l,W),\end{equation}
where $\eps>0$ is sufficiently small, $W$ is a neighborhood of $\Coin(A,F;U)$ (in $K$)  and $g^\eps_\l$ is given by \eqref{cons6} with  $f\in a(F,\eps)$. Arguments from Steps 1 -- 7 justify this construction and show that the sequence in the right hand of the definition stabilizes and its limit does not depend on any auxiliary objects used to define it.

\begin{Th}\label{main1} Let a pair $(A,F)$ be admissible. The degree defined by \eqref{def-deg} has the following properties:
\vspace{-4mm}
\begin{enumerate}
\item \emph{(Existence)} If\, $\deg_K(A,F;U)\neq 0$ then $\Coin(A,F;U)\neq\emptyset$.
\item \emph{(Additivity)} If\, $U_1$, $U_2\subset U$ are open disjoint and $\Coin(A,F;U)\subset (U_1\cup U_2)\setminus \ov{U_1\cap U_2}$, then
    $$\deg_K(A,F; U)=\deg_K(A,F;U_1) +\deg_K(A,F;U_2).$$
\item \emph{(Homotopy invariance)} If\, $(\{A(t)\}_{t\in [0,1]},\Phi)$ is an admissible homotopy, then
    $$\deg_K\big(A(0),F(\cdot,0);U\big)= \deg_K\big(A(1),F(\cdot,1);U\big).$$
\item \emph{(Normalization)} If\, $K$ is bounded, $F\colon K\multi  E$ and $F(K)$ is bounded in $ E$, then the Euler characteristic $\chi(K_A)$, where $K_A:=K\cap \ov{D(A)}$, is well-defined and $\deg_K(A,F;K)=\chi(K_A).$

\end{enumerate}
\end{Th}
\begin{proof} (1) Suppose to the contrary that $\Coin(A,F;U)=\emptyset$ and take open $W\subset K$ and $\bar\eps>0$ as in Step 1. Arguing as in Step 2 we get $0<\eps_0\leq\bar\eps$ such that $\big\|v-f(u)\big\|\geq L\eps$ for any $u\in\ov W\cap D(A)$ and $v\in Au$, where $f\in a(F,\eps)$ with $0<\eps<\eps_0$. If $0<\eps<\eps_0$ is small enough and $\l_n\searrow 0$, then $0\neq\deg_K(A,F;U)=\ind_K(g_n,W)$, where $g_n(u):=J_{\l_n}\circ r\big(u+\l_nf(u)\big)$, $u\in\ov W$, and $f\in a(F,\eps)$. Arguing as in Step 3 we find sequences $(u_n)$ in $W$, $v_n\in Au_n$ such that
(after passing to a subsequence) $u_n\to u_0$ and
$$L\eps_0\leq\limsup_{n\to\infty}\big\|v_n-f(u_n)\big\|\leq \limsup_{n\to\infty} \left(L\frac{d_K\big(u_n+\l_nf(u_0)\big)}{\l_n}+L\big\|f(u_n)-f(u_0)\big\|\right)<L\eps:$$
a contradiction.\\
\indent (2) The additivity property follows immediately from the additivity property of the index $\ind_K$.\\
\indent (3) We argue as follows. Choose an open $W\subset K$ and $\bar\eps>0$ such that $\Phi\big(\ov W\times [0,1]\big)$ and $f\big(\ov W\times [0,1]\big)$ are bounded, where $f\in a(\Phi,\eps)$ with $0<\eps\leq\bar\eps$ and $\bigcup_{t\in [0,1]}\Coin\big(A(t),\Phi(\cdot,t);U\big)\subset W\subset \ov W\subset U$. Using arguments similar to those from Step 2 we get $\eps_0\leq\bar\eps$ such that $\big\|v-f(u,t)\big\|\geq L\eps_0$ for any $t\in [0,1]$, $u\in\part W\cap D(A(t))$ and $v\in A(t)u$, where $f\in a(\Phi,\eps)$, $0<\eps<\eps_0$. Next we show that there is $\l>0$ such that for every $f\in a(F,\eps)$, $0<\eps<\eps_0$, the set $\Big\{u\in\ov W\mid u=J_\l\circ r\big(u+\l f(u,t)\big),\;\text{for some}\; t\in [0,1]\Big\}\subset W$: this can be done by exactly the same arguments as in Step 3. Take a small $\eps>0$ and  $f\in a(\Phi,\eps)$ satisfying condition \eqref{+app}. Hence
$f\colon \ov W\times [0,1]$ is an admissible (see page \pageref{index}) homotopy joining $f(\cdot,0)\in a\big(\Phi(\cdot,0),\eps\big)$ to $f(\cdot,1)\in a\big(\Phi(\cdot,1),\eps\big)$.
Therefore
$$\deg_K\big(A(0),\Phi(\cdot,0);U\big)=\ind_K\big(g(\cdot,0),W\big)=
\ind_K\big(g(\cdot,1),W\big)=\deg_K\big(A(1),\Phi(\cdot,1),U\big),$$
where $g(u,t):=J_\l\circ r\big(u+\l f(u,t)\big)$, $u\in \ov W$, $t\in [0,1]$, with sufficiently small $\l>0$.\\
\indent (4) By definition $\deg_K(A,F,K)=\ind_K(g,K)$, where $g\colon K\to K$ is compact, $g(u)=g_\l(u):=J_\l\circ r(u+\l f(u))$, $u\in K$, where $f\in a(F,\eps)$ and $\eps,\l>0$ are sufficiently small. Actually $g(K)\subset K_A$. Let $g'\colon K\to K_A$ be the contraction of $g$, $j\colon K_A\to K$ be the inclusion and $g_A:=g'|_{K_A}=g'\circ j_A$. Clearly $g_A$ is compact. Consider  $h\colon K_A\times [0,1]\to K$ given by $$
h(u,t):= \left\{
\begin{array}{ll}
g_{t\l}(u) & u\in K_A,\, t\in (0,1] \\
u  & u\in K_A, \, t=0.
\end{array}
\right.
$$
In view of (c) on page \pageref{accr}, we see that $h$ is a well-defined continuous homotopy joining $g_A$ to the identity $\id_{K_A}$.  In view of the normalization property of $\ind_K$, $g$ is a Lefschetz map and
$\Lambda(g)=\ind_K(g,K)$. The commutativity of the diagram
\vspace{-3mm}
$$\xymatrix@R-10pt{K\ar[r]^-{g'}\ar[d]^-{g}&{K_A}
\ar@<-2pt>[d]^-{g_A}
\ar[dl]_-j\\
K\ar[r]_-{g'}&{K_A.}}$$
along with \cite[Lemma (3.1)]{Granas} implies $g_A$ is a Lefschetz map, i.e. $H_*(g_A)$ is a Leray endomorphism (see \cite[Section 2]{Granas})  and $\Lambda(g_A)=\Lambda(g)$ (\footnote{$H_*(\cdot)$ stands for the singular homology functor with the rational coefficients.}). Since $g_A$ is homotopic to $\id_{K_A}$, we have that $H_*(g_A)=H_*(\id_{K_A})=\id_{H_*(K_A)}$. Therefore $\id_{H_*(K_A)}$ is a Leray endomorphisms , i.e., the graded vector space $H_*(K_A)$ is of finite type and the Euler characteristic $\chi(K_A):=\sum_{q\geq 0}(-1)^q\dim_\Q H_q(K_A)$ is a well-defined integer number. Moreover $\chi(K_A)=\l(\id_{K_A})=\Lambda(\id_{K_A})=\Lambda(g_A)=\Lambda(g)$, where $\lambda(\id_{K_A})$ is the ordinary Lefschetz number.\end{proof}
Let us now derive a series of result that can be treated as constrained generalizations of the Schaeffer or Leray-Schauder nonlinear alternatives.
\begin{Prop}\label{LS1} In addition to our standing assumptions, suppose that $U=K$ (i.e., $F\colon K\multi E$), $K$ is closed convex, $A$ is densely defined or $E$ is uniformly convex. Let $C:=\{u\in K\cap D(A)\mid Au\cap tF(u)\neq\emptyset\; \text{for some}\; t\in [0,1]\}$. If $F(C)$ is bounded and\\
\indent (a) $K$ is bounded; or\\
\indent (b) $C$ is bounded, $A$ is $\omega$-$m$-accretive with $\omega\leq 0$ and $0\in Au_0$ for some $u_0\in D(A)$,\\
then $\deg_K(A,F;K)=1$.
\end{Prop}
\begin{proof} Under our assumption $\ov{D(A)}=E$ or, if $E$ is uniformly convex (recall that so is $E^*$ by assumption), then  $\ov{D(A)}$ is convex in view of \cite[Proposition 3.5]{Barbu}.  Let $\Phi\colon K\times [0,1]\multi E$ be given by $\Phi(u,t):=tu$, $u\in K$. It is easy to see that $\Phi$ is admissible. Hence by the homotopy invariance and normalization properties we get $\deg_K(A,F;K)=\deg_K(A,0;K)=\chi(K_A)=1$ because $K_A$ is convex.\\
\indent Suppose now that $K$ is not bounded but $C$ is. The map $\Phi$ defined above provides an admissible homotopy showing that $\deg_K(A,F;K)=\deg_K(A,0;K)$.
Take $R>0$ such that $C\subset B(u_0,R)$, let $W:=B(u_0,R)\cap K$ and $K':=\ov W=K\cap D(u_0,R)$. Then $K'$ is closed convex; hence $K'\in \ANR$. By definition
$\deg_K(A,0;K)=\ind_K(J_\l,W)$ for sufficiently small $\l>0$. Since $J_\l$ is nonexpansive and $J_\l(u_0)=u_0$ we get $J_\l(K')=J_l(\ov W)\subset\ov W=K'$. Hence, by the contraction and normalization properties of the index we get that
$$\ind_K(J_\l,W)=\ind_{K'}(J_\l,W)=\ind_{K'}(J_\l,K')=\chi\big(K'\cap\ov{D(A)}\big)=1.\eqno\qedhere$$
\end{proof}
As a consequence we get the counterpart of the local normalization property of the index.
\begin{Cor}\label{loc-norm} Suppose that $A$ is $\omega$-$m$-accretive with $\omega<0$, $A$ is densely defined or $E$ is uniformly convex, $K\subset E$ is closed convex. If $0\in Au_0$ for some $u_0\in D(A)$, then for any $U$ open in $K$,
$$\deg_K(A,0;U)=\begin{cases}1&\text{if}\;\; x_0\in U,\\
0&\text{otherwise}.\end{cases}$$
\end{Cor}
\begin{proof} For the proof it is sufficient to observe that $A$ is invertible since $A=-\omega(I+(-\omega)^{-1}(A+\omega I)$ and appeal to part (b) of the above Proposition.
\end{proof}
\begin{Prop}\label{LS2} Again, in addition to the standing assumptions, let $K\subset E$ be closed convex, $U=K$, $A$ an $\omega$-$m$-accretive operator with $\omega\leq 0$, $0\in Au_0$ for some $u_0\in D(A)$ and $A$ is densely defined on $E$ is uniformly convex. Suppose that $F$ is bounded on bounded sets and
$$\limsup_{\|u\|\to\infty,\;u\in K}\sup_{v\in F(u)}[u-u_0,v]_+<0.$$
Then $C=\Coin(A,F;K)\neq\emptyset$ is unbounded or $\deg_K(A,F;K)=1$.
\end{Prop}
\begin{proof} Assume that $C$ is bounded and take $R>0$ such that $C\subset W:=B(u_0,R)$ and $\sup_{v\in F(u)}[u-u_0,v]_+\leq 0$ for all $u\in E$ with $\|u-u_0\|=R$. This implies that $F(u)\subset  T_{D(u_0,R)}(u)$ for any $u\in D(u_0,R)$. If $R$ is large enough, then $B(u_0,R)\cap K\neq\emptyset$ and, by \cite[Theorem 4.1.16]{A-E}, $T_{K\cap D(u_0,R)}(u)=T_K(u)\cap T_{D(u_0,R)}(u)$ for all $u\in K\cap D(u_0,R)$. Therefore $(A,F)$ is admissible with respect to the $\cal L$-retract $K':=K\cap D(0,R)$. Observe also that since $J_\l$ is nonexpansive, we get $J_\l(K')\subset K'$ for any $\l>0$.  Therefore we are in a position to define $\deg_K(A,F;K)=\deg_K(A,F;W)$ and $\deg_{K'}(A,F;K')$. By the normalization property we see that $\deg_{K'}(A,F;K')=1$. We now show that both degree are equal. To see this take an tangent $\eps$-approximation of $f\colon K\to E$ of $F$. If $\eps>0$ is small enough that
$deg_{K'}(A,F;K')=\ind_{K'}(g'_\lambda,K')$ where $g'_\l(u)=J_\l\circ r'(u+\l f(u))$ for $u\in K'$, $r'$ is an $\cal L$-retraction (defined on  $E$) onto $K'$ and $\l.0$ sufficiently small.
On the other hand $\deg_K(A,F;W)=\ind_K(g_\l,W)$, where $g_\l(u)=J_\l\circ r(u+\l f(u))$, $u\in\ov W$. Consider the map $h\colon \ov W\times [0,1]\to K$ given by
$$h(u,t)=J_\l\circ \bigg(u+\l f(u)+t\Big(r'\big(u_\l f(u)\big)-\big(u+\l f(u)\big)\Big)\bigg),\; u\in \ov W,\;t\in [0,1].$$
Exactly as in Step 6 of our construction we show that $h(u,t)\neq u$ for $u\in\part W$ and $t\in [0,1]$. Hence $$\deg_K(A,F;W)=\ind_K(h(\cdot,1),W),$$ where $h(\cdot,1)\colon \ov W\to K$ is given by $h(u,1)=J_\l\circ r'(u+\l f(u))$. This, in view of the contraction property of the index, concludes the proof. \end{proof}
\subsection{Weakly upper semicontinuous perturbations} Now, in addition to Assumptions \ref{assad} and \ref{assadhom}, we assume that
\vspace{-4mm}
\begin{enumerate}
\item[$(D_4)$] $A$ is a densely defined linear operator (see  (b) on page \pageref{accr});
\item[$(D_5)$] $(A,F)$ is admissible and $F\colon U\multi E$ is weakly usc with convex weakly compact values.
\end{enumerate}
\vspace{-2mm}
Recall that $A$, as the generator of a $C_0$ semigroup, is closed and densely defined.

\noindent {\sc The construction:} The idea of this construction is similar to that above. However under new assumptions tangent graph-approximations are not available. Therefore we consider \eqref{problem3} with  $f$ being an `approximation' of a different type (comp. \cite{Lasry, K-Diss}). Namely, it is possible to find a continuous field $q(u)\in E^*$, $u\in U$, such that if $Au\not\in F(u)$, then $q(u)$ separates $Au$ from $F(u)$, i.e., $Au$ and $F(u)$ lie in different half-spaces determined by the hyperplane induced by $q(u)$, and $f$ is a continuous map such that $f(u)$ lies in the same half-space as $F(u)$ does.\\
\indent As above we proceed in several steps.

\noindent {\em Step 1}: Take a bounded open (in $K$) set $W$ such that
\begin{equation}\label{cons8}C:=\Coin(A,F;U)=\big\{u\in U\cap D(A)\mid Au\in F(u)\big\}\subset W\subset\ov W\subset U\end{equation} and $F$ is bounded on $\ov W$, i.e., $\sup_{v\in F(u), u\in\ov W}\|v\|<\infty$.

\noindent {\em Step 2}: We need a technical lemma.
\begin{Lem}\label{sep} There are bounded continuous maps $q\colon \ov W\to E^*$, $w=w_q\colon \ov W\to\R$ and $\eps_0>0$  such that:
\begin{equation}\label{sep1}\begin{split} &w(u)=\big\la Au,q(u)\big\ra\; \text{for}\;  u\in \ov W\cap D(A),\;\;q(u)=0\;\; \text{for}\;\; u\in C\;\;\text{and}\\
&\inf_{v\in F(u)}\big\la v,q(u)\big\ra>w(u)\; \text{for}\; u\in W\setminus C,\\
&\inf_{v\in F(u)}\big\la v,q(u)\big\ra>w(u)+\eps_0\;\; \text{for}\; u\in\part W.
\end{split}\end{equation}
For any $0<\eps<\eps_0$ and $\gamma>0$, there is a continuous bounded  map $f\colon \ov W\to E$ such that for $u\in\part W$
\begin{equation}\label{sep2} d_K^\circ\big(u,f(u)\big)<\gamma\eps\;\;\text{and}\;\;\big\la f(u),q(u)\big\ra>w(u)+\eps.\end{equation}
\end{Lem}
\noindent The map $q$ separates $Au$ and $F(u)$ if $u\not\in C$ (see second and third condition in \eqref{sep1}). Condition \eqref{sep2} means that $f$ is an `approximation' (in the above mentioned sense) and is `almost' tangent to $K$.
\begin{proof}
Define $G(u):=J_0\big(u+\l_0 F(u)\big)-u$, $u\in \ov W$, where $J_0:=J_{\l_0}$ with $\l_0>0$ such that $\l_0(\omega+1)\leq1$. Since $J_0$ is linear, it is weak-to-weak continuous. Therefore $G$ is weakly usc and has convex weakly compact values. Moreover $\big\{u\in \ov W\mid 0\in G(u)\big\}=C:=\Coin(A,F;U)$. Hence for any $u\in \ov W\setminus C$, $0\not\in G(u)$, i.e., $\inf_{v\in G(u)}\|v\|>\beta_u>0$. As it is not difficult to see, there is $\eps_0>0$ such that $\inf_{v\in G(u)}\|v\|>\eps_0$ for $u\in \part W$ (in the proof, similar to that provided in Step 2 of the previous construction, the `strong$\times$weak' closedness of $\Gr(A)$ plays a role).  Given $u\in \ov W\setminus C$, we have
$$\inf_{v\in G(u)}\|v\|=\inf_{v\in G(u)}\sup_{q\in E^*,\|q\|\leq 1}\la v,q\ra=\sup_{q\in E^*,\|q\|\leq 1}\inf_{v\in G(u)}\la v,q\ra$$
in view of the Sion version of the von Neumann min-max equality. Hence for any $u\in\ov W$ there is $q_u\in E^*$, $\|q_u\|\leq 1$, such that  $\inf_{v\in G(u)}\la v,q_u\ra>\beta_u$ and, for $u\in\part W$,
$\inf_{v\in G(u)}\la v,q_u\ra>\eps_0$. For $u\in W$, let $V(u):=\{y\in W\setminus C\mid \inf_{v\in G(y)}\la y,q_u\ra>\beta_u\}$ and for $u\in\part W$ let $V(u)=\{y\in \ov W\setminus C\mid\inf_{v\in G(y)}\la v,q_u\ra>\eps_0\}$. Evidently $\big\{V(u)\big\}_{u\in \ov W}$ is an open covering of $\ov W$ since $G$ is uhc (see page \pageref{set-val}). Let $\{\l_s\}_{s\in S}$ be a partition of unity subordinated to this cover, i.e., for any $s\in S$, there is $u_s\in \ov W\setminus C$ such that $\supp\l_s\subset V(u_s)$. Additionally take a continuous $\mu\colon \ov W\to [0,1]$ such that $C=\mu^{-1}(0)$, $\part W=\mu^{-1}(1)$ and define $q_0\colon U\to E^*$ by
$$q_0(u):=\begin{cases}\mu(u)\sum_{s\in S}\l_s(u)q_{u_s}& \text{for}\;u\in \ov W\setminus C,\\
0& \text{for}\;u\in C.\end{cases}$$
It is easy to see that $q_0$ is continuous and so is $q\colon U\to E^*$ given by $q:=\l_0 J_0^*\circ q_0$; evidently $\big\|q(u)\big\|\leq\l_0\|J_0^*\|\big\|q_0(u)\big\|\leq 1$ for $u\in\ov W$. Let $w\colon U\to\R$ be given by $w(u):=\big\la u-J_0(u),q_0(u)\big\ra, u\in U$. Then $w$ is continuous,
$$w(u)=\Big\la J_0\big((u+\l_0 Au)-u\big),q_0(u)\Big\ra=\big\la Au,q(u)\big\ra\;\; \text{for}\;\; u\in D(A)\cap \ov W,$$
for $u\in \ov W\setminus C$
\begin{align*}\inf_{v\in F(u)}\big\la v, q(u)\big\ra-w(u)\geq &\;\mu(u)\sum_{s\in S}\l_s(u)\inf_{v\in F(u)}\big\la J_0(u+\l_0 v)-u,q_{u_s}\big\ra>
\mu(u)\sum_{s\in S}\l_s(u)\beta_{u_s}>0\end{align*}
and $\inf_{v\in F(u)}\big\la v, q(u)\big\ra-w(u)>\eps_0$  for $u\in\part W$.\\
\indent Take any $\gamma>0$ and $0<\eps<\eps_0$. For $u\in\part W$ take $v_u\in F(u)\cap T_K(u)$. Then $\big\la v_u,q(u)\big\ra>w(u)+\eps_0>w(u)+\eps$. Let $T(u):=\big\{y\in\part W\mid
d_K^\circ(y,v_u)<\gamma\eps,\; \big\la v_u,q(y)\big\ra>w(y)+\eps\big\}$. Then $\{T(u)\}_{u\in\part W}$ is an open cover of $\part W$. Le $\{\lambda_s\}_{s\in S}$ be a partition of unity subordinated to this cover, i.e., for $s\in S$ there is $u_s\in\part W$ such that $\supp\lambda_s\subset T(u_s)$. Define $\bar f\colon \part W\to  E$
$$\bar f(u)=\sum_{s\in S}\lambda_s(u)v_{u_s},\; u\in\part W.$$
Then $\bar f$ is continuous, bounded and for any $u\in\part W$
$$\big\la \bar f(u),q(u)\big\ra=\sum_{s\in S}\lambda_s(u)\la v_{u_s},q(u)\ra>w(u)+\eps,\;\; d_K^\circ\big(u,\bar f(u)\big)<\gamma\eps.$$
Finally let $f\colon \ov W\to E$ be an arbitrary continuous bounded extension of $\bar f$.\end{proof}
\noindent {\em Step 3}: Take $\gamma=\frac{1}{L}$, where $L$ is the constant from the definition of an $\cal L$-retract, $0<\eps<\eps_0$ and $f$ satisfying condition \eqref{sep2}. For any $\l>0$ with $\l\omega<1$ define $g_\l\colon \ov W\to K$ by
$$g_\l(u):=J_\l\circ r\big(u+\l f(u)\big),\;\; u\in \ov W,$$
where $r\colon B(K,\eta)\to K$ be an $\cal L$-retraction (with the constant $L$). It is clear that wlog we may assume that this definition is correct for $0<\l<\bar\l$ (see Assumption \ref{assad} (1)). We claim that there is $\l_0>0$, $\l_0\omega<1$ such that for $0<\l<\l_0$,
$$\Fix g_\l=\big\{u\in\ov W\mid u=g_\l(u)\big\}\subset W.$$
\indent Suppose to the contrary. As in Step 3 of the previous construction we get sequences $\l_n\to 0$, $0<\l_n<\bar\l$, $u_n\in\part W\cap D(A)$ such that (after passing to a subsequence) $u_n\to u_0\in\part W$ and
$$\big\|Au_n-f(u_n)\big\|\leq L\frac{d_K\big(u_n+\l_nf(u_n)\big)}{\l_n}\leq L\frac{d_K\big(u_n+\l_nf(u_0)\big)}{\l_n}+L\big\|f(u_n)-f)u_0)\big\|.$$
On the other hand
$$\eps_0<\big\la f(u_n),q(u_n)\big\ra-w(u_n)=\big\la f(u_n),q(u_n)\big\ra-\big\la Au_n,q(u_n)\big\ra\leq \sup_{u\in\ov W}\big\|q(u)\big\|\big\|f(u_n)-Au_n\big\|\leq \|f(u_n)-Au_n\|.$$
Passing with $n\to\infty$ we get
$$\eps\leq \limsup_{n\to\infty}\big\|Au_n-f(u_n)\big\|\leq Ld_K^\circ\big(u_0,f(u_0)\big)<\eps.$$
A contradiction concludes the proof.

\noindent {\em Step 4}: We are in a position to define $\ind_K(g_\l,W)$ if $\l$ and $g_\l$ are as above and then let
\begin{equation}\label{def-deg-2}d_K(A,F;U)=\lim_{\l\to 0}\ind_K(g_\l,W).\end{equation}
It is relatively easy to show, by the use of arguments similar to those used in the previous construction, that this definition is correct since it does not depend on the choice of $W$ satisfying condition \eqref{cons8}, a separating map $q$ satisfying condition \eqref{sep1}, `approximation' $f$ satisfying condition \eqref{sep2} and an $\cal L$-retraction $r$.
\begin{Th}\label{main2} The function $\deg_K$ defined by \eqref{def-deg-2} has the properties enlisted in Theorem \ref{main2}. We leave the detailed formulation to a reader.
\end{Th}
\begin{proof} One follows (at least from the conceptual viewpoint) the arguments of the proof of Theorem \ref{main1}. For instance in order to get the existence we assume to the contrary that $C=\emptyset$ and construct a separating map $q$ and an `approximation' $f$ such that conditions \eqref{sep1}, \eqref{sep2} hold on the whole $\ov W$. In this case, as it is easy to see, $\ind_K(g_\l,W)=0$ for all sufficiently small $\l$.
\end{proof}
\begin{Rem} {\em Let us now observe that if $A$ is a resolvent compact linear quasi-$m$-accretive operator and $F$ is weakly tangent and $H$-usc, then two degree theories are available. The degree $\deg_K(A,F;U)$ may be defined via formulae \eqref{def-deg} or \eqref{def-deg-2}. It can be easily seen by observing that if $f$ is a tangent $\eta$-(graph)-approximation of $F$, then the condition \eqref{sep2} is satisfied provided that $\eta>0$ is sufficiently small.\\
\indent In particular if $A$ is $m$-accretive (resp. linear and $m$-accretive) and  $F$ is a single-valued map, then the degree defined via \eqref{def-deg} (resp. \eqref{def-deg-2}) coincide with the degree considered in \cite{b:C-K2}.
Hence, apart from from results being direct consequences of Propositions \ref{LS1}, \ref{LS2} and Corollary \ref{loc-norm}, we get the following}\end{Rem}

\begin{Prop} {\em (see \cite[Proposition 4.2]{b:C-K2})} Assume that $K$ is a closed convex cone, $\lambda_1>0$ is the smallest
real eigenvalue of $A$ to which there corresponds an eigenvector
$u_1\in K\setminus\{0\}$ such that $(A-\lambda I)^{-1}u_1\cap
K=\emptyset$ and $\ker (A-\lambda I)\cap K=\{0\}$ for all
$\lambda>\lambda_1$. Then
$$\deg_K(A,\l I,K)=\begin{cases}1&\text{if}\;\;\l<\l_1\\
0&\text{if}\;\; \l>\l_1. \end{cases}\eqno\square$$
\end{Prop}

\section{Nonlinear reaction-diffusion equation}

We now we shall apply the introduced degree to discuss the existence of solutions to  problem \eqref{reac-diff}, i.e. the $M$-dimensional system
\begin{equation}\label{reac-diff1}
\left\{\begin{array}{l} -\Delta(\rho\circ u)(x)\in\vp\big(u(x)\big),\; u=(u_1,...,u_M)\in\R^M,\;x\in\Omega,\\
u_i(x)\geq 0,\;x\in\Omega,\;i=1,...,M,\\
u|_{\partial\Omega}=0,\end{array}\right.
\end{equation}
$\Omega\subset\R^N$ is open bounded with smooth boundary $\partial\Omega$ and $\Delta$ denotes the vectorial Laplace operator (\footnote{If a function $v=(v_1,...,v_M)\colon \Omega\to\R^M$ is twice differentiable (or $v\in H^2(\Omega,\R^M)$), then $\Delta v=(\Delta v_1,...,\Delta v_M)$ and $\Delta v_i=\sum_{j=1}^N\part^2_{jj}v_i$, $i=1,...,M$.}).
\begin{Ass}\label{rd-system} {\em Let us make the following standing assumptions.
\begin{enumerate}
\item $\vp\colon\R^M_+\multi\R^M$ is usc with compact convex values and sublinear linear growth, i.e., for some $c>0$
\begin{equation}\label{fi-growth}\sup_{v\in \vp(u)}\|v\|\leq c\big(1+|u|\big),\ \ u\in\R^M_+;\end{equation}
\item for any $u\in \R^M_+$ there is $v\in\vp(u)$ such that $v_i\geq 0$ if $u_i=0$;
\item $\rho=(\rho_1,...,\rho_M)\colon\R^M_+\to\R^M_+$ is a homeomorphism and $\big|\rho(u)\big|\geq \alpha |u|$ for some $\alpha>0$;
\item For any $I\subset\{1,...,M\}$ the face $\{y\in\R^M_+\mid y_i=0\; \text{for}\; i\in I\}$ is invariant with respect to $\rho$.
\end{enumerate}
}\end{Ass}
We say that a function $u\colon\Omega\to\R^M_+$ is a ({\em weak}) {\em solution} to \eqref{reac-diff1} if $\rho\circ u\in H^1_0(\Omega,\R^M)$ and there is a function $v\in L^2(\Omega,\R^M)$ such that $v(x)\in \vp(u(x))$ a.e. on $\Omega$ and for any $\phi\in H_0^1(\Omega,\R^M)$
$$\int_\Omega\big\la v(x),\phi(x)\big\ra_{\R^M}\,\d x=\int_\Omega\sum_{i=1}^M\big\la\nabla \rho_i(u(x)),\nabla\phi_i(x)\big\ra_{\R^N}\,\d x.$$
It is easy to see that if $u$ is a solution, then $u\in L^2(\Omega,\R^M)$ and $u|_{\part\Omega}=0$ in the sense that there is a sequence $(u_n)$ in $C^\infty_0(\Omega,\R^M)$ such that $\|u_n-u\|_{L^2(\Omega,\R^M)}\to 0$.
\begin{Rem}{\em In case of single-valued mapping $\vp$, i.e., if $\vp=(\vp_1,...,\vp_M)$ with $\vp_j\colon \R^M\to\R$, problem \eqref{reac-diff1} is a reaction-diffusion system describing  equilibria  of the distribution of $M$ substances subject to chemical reactions and diffusion. The above problem corresponds to the situation of the problem under control, with uncertainties  or discontinuities. The constraint $u_j(x)\geq 0$, $x\in\Omega$, $1\leq j\leq M$ is physically justified and means that the concentration $u_i$ of the $i$-th reactant  is nonnegative. It is worth pointing out here, that contrary to many other results, we avoid the assumption of `nonnegativity' of $\varphi$, implying that all substances are only produced. This is unrealistic since during the process some reactants vanish or are transformed. Instead we assume \ref{rd-system} (2).  It actually means that if some reactant vanishes in some area, its amount (in this area) cannot decrease.
The reaction-diffusion problem was studied by many authors,  see e.g. \cite{Bothe} and references therein,  but the case under constraints is still not very well recognised. In \cite{b:C-K2}, (see also \cite{C-M}) problem \eqref{reac-diff1} was considered in a one-dimensional case, i.e. if $M=1$, and with a single-valued $\varphi$ is single-valued. The situation $M>1$ with a single-valued perturbation was treated in \cite{M}.
}\end{Rem}

Let $E=L^2(\Omega,\R^m)$ and let
$$K:=\big\{u\in E\mid u_i(x)\geq 0\;\; \text{for a.a.}\;\; x\in\Omega\big\}.$$
By $A\colon E\supset D(A)\to E$ we denote the self-adjoint  Dirichlet $L^2$-realization of the classical vectorial  Laplacian $-\Delta$ (cf. e.g., \cite[Theorem 4.27]{Grubb}), i.e.,  it arises from the Lions-Lax-Milgram construction (cf. \cite[Theorem 12.18]{Grubb}) and let $F\colon K\to E$ be the Nemytskii operator  $N_{\vp\circ\gamma}$ determined by $\vp\circ\gamma$, i.e.,
$$F(u)=N_{\vp\circ\gamma}(u):=\big\{v\in E\mid v(x)\in \vp\circ \gamma\big(u(x)\big)\; \text{for a.a.}\; x\in\Omega\big\}, $$
where $\gamma:=\rho^{-1}\colon \R^M_+\to\R^M$. It is clear that $u\in E$ solves \eqref{reac-diff1}  iff $w:=\rho(u)\in E$  is a solution of the problem \begin{equation}\label{eq:Abstr Incl intro}Aw\in F(w),\;\; w\in K.
\end{equation}
\begin{Prop} The following conditions are satisfied:
\vspace{-4mm}
\begin{enumerate}
\item [$(a)$] $A$ is a resolvent compact linear densely defined $m$-accretive operator in $E$, $D(A)=H_0^1(\Omega,\R^M)\cap H^2(\Omega,\R^M)$;
    \item [$(b)$] $K$ is resolvent invariant, i.e., $J_\l^A(K)\subset K$ for any $\l>0$;
    \item [$(c)$] $F$ is $H$-usc with weakly compact convex values and weakly tangent to $K$.
\end{enumerate}
\end{Prop}
\begin{proof} It is well-known that $-A$ generates a $C_0$ semigroup of contractions, hence $A$ is $m$-accretive.  The compactness of resolvents follow immediately from the compactness of the Sobolev embedding $H_0^1(\Omega,\R^M)\subset E$. The invariance follows from \cite[Proposition 4.3]{KS}.\\
\indent  It is evident that values of $F$ are convex and closed and bounded, hence, weakly compact. The $H$-upper semicontinuity of $F$ follows from assumptions \ref{rd-system} (1), (3) and  \cite[Lemma 4.2]{KS}.\\
\indent It is clear that if $u\in\R^M$, then $T_{\R^M_+}(u)=\{v\in\R^M\mid v_j\geq 0\;\; \text{if}\;\; u_j=0,\;j=1,...,M\}$.
Conditions (2) and (4) from assumption \ref{rd-system} imply that $\vp\circ\gamma(u)\cap T_{\R^M_+}(u)$ for any $u\in\R^M_+$. Let now $u\in K$. The map $T_{\R^M_+}(\cdot)\colon \R^M_+\multi\R^M$ is lower semicontinuous (cf. \cite[Th. 4.2.2]{A-F}), $\vp\circ\gamma$ is measurable; hence $\Omega\ni x\multi \vp\circ \gamma(u(x))\cap T_{\R^M_+}(u(x))\subset\R^M$ is measurable with nonempty values. By  the Kuratowski--Ryll-Nardzewski theorem, there is a measurable  $v\colon \Omega \to \R^M$ such that $v(x)\in \vp\circ\gamma\big(u(x)\big)\cap T_{\R^M_+}\big(u(x)\big)$ for a.e. $x\in \Omega$. Clearly $v\in E$ and $v\in T_K(u)\cap F(u)$ since in view of \cite[Cor. 8.5.2]{A-F} $T_K(u)=\big\{v\in E\mid v(x)\in T_{\R^M_+}\big(u(x)\big)\;\hbox{a.e. in}\;\Omega\big\}$.
\end{proof}

\begin{Ex}\label{ex:Nemytskii not usc} {\em Let $\varphi\colon\R_+\to\R$, $\varphi(u)=[u-1,u+1]$, $u\in\R_+$. For any $n\in\N$, $s_n\in N_\vp(0)$ (the Nemytskii operator associated to $\vp$), where $s_n(x)=\sin(nx)$, $x\in \Omega=(-\pi/2,\pi/2)$, but $(s_n)$ has no $L^2$-convergent subsequence. Moreover
$C=\big\{(1+1/n)s_n\mid n\in\N\big\}$ is closed in $L^2(\Omega)$ but $N_\varphi^{-1}(C)$ is not. This shows that neither $N_\vp$ is usc nor it has compact values.}
\end{Ex}

\subsection{Existence}\label{linearisation}  In what follows we are going to find conditions for the nontriviality of the difference $$\deg_K(A,F;U_0)-\deg(A,F,U_\infty),$$ where $U_0$ (resp. $U_\infty$)  is a `small' (resp. `large') ball around 0 and, thus, establishing the existence of nontrivial solutions to problem.

\noindent {\sc Linear perturbations:} Let $G\colon E\to E$ be the Nemytskii operator determined by a  matrix $D=[d_{ij}]_{i,j=1}^M$, i.e., $G(u)(x)=Du(x)$ for $u\in E$ and $x\in \Omega$. From \cite[Example 3.1]{Ka-K} it follows that $Du\in T_{\R^M_+}(u)$ for  $u\in\R^M_+$ iff $D$ is {\em quasi-nonnegative}, i.e., the off-diagonal entries $d_{ij}\geq 0$, $1\leq i\neq j\leq M$.\\
\indent We therefore assume that $D$ is quasi-nonnegative. Hence $G(u)\in T_K(u)$ for any $u\in K$.  Let $$\splus(D):=\{\lambda\in\R\mid\textrm{there exists}\; 0\neq u\in\R^M_+\textrm{ such that }Du=\lambda u\}\subset\sigma(D).$$
In view of \cite [Page 2241]{Ka-K} it immediately follows that  $\splus(D)$ is nonempty and $$s(D):=\max\Re\big(\sigma(D)\big)=\max\splus(D).$$
Let $\lambda_1$ be the first eigenvalue of the (scalar) Laplace operator $-\Delta$ on $\Omega$ with Dirichlet boundary condition and $\phi\geq 0$ be an eigenfunction related to $\lambda_1$. Then $\phi>0$ on $\Omega$ and $\lambda_1$ is simple, i.e., $\phi$ is unique up to a positive factor (cf., e.g.,  \cite{Evans} for details).

\begin{Prop}\label{war-spektrum} If  $\lambda_1\not\in\splus(D)$, then
$\ker(A-G)\cap K=\{0\}$. If, additionally, $s(D)>\lambda_1$, then there is $\bar u\in K$ such that
$(A-G)^{-1}(\bar u)\cap K=\emptyset$.
\end{Prop}
\begin{proof} {\em Step 1}: Let $E_1:=\{u\in E\mid Au=\l_1u\}$ be the eigenspace corresponding to $\l_1$ and let $E_1^\bot$ be the orthogonal complement of $E_1$. Clearly both $E_1$ and $E_1^\bot$ are invariant with respect to $A$; a direct computation shows also that so they are with respect to $G$. Let us collect some other properties:\\
\indent (1) $E_1^\bot\cap K=\{0\}$. For if $u\in E_1^\bot\cap K$ then for every $j=1,\ldots,M$ we have $\int_\Omega u_j\phi \d x=\la u,\phi e_j\ra_E=0$, where $e_j\in\R^M$ is the $j$-th vector from the standard basis of $\R^M$. Since the integrand is nonnegative, thus  it is zero almost everywhere. Therefore $u_j=0$, as $\phi$ is strictly positive on $\Omega$.\\
\indent (2) If $u=u^1+u^\bot\in K$, where $u^1\in E_1$, $u^\bot\in E_1^\bot$, then $u^1\in K$. To see this we show that for $j=1,\ldots,M$,
$0\leq \la u,\phi e_j\ra_E=\la u^1,\phi e_j\ra=\int_\Omega u_j^1\phi\d x$. Now, since $u^1_j=\alpha_j\phi$ for some $\alpha_j\in\R$, we get $\alpha_j\geq 0$, which yields $u^1_i\geq 0$, and consequently $u^1\in K$.\\
\indent (3) If $\lambda_1\not\in\splus(D)$, then $\ker(A-G)\cap K\cap E_1=\{0\}$. Indeed : let $u\in \ker(A-G)\cap K\cap E_1$. Therefore $G(u)=Au=\lambda_1 u$. If $u\not\equiv 0$ then there is $x\in \Omega$ such that $u(x)$ is a nonnegative nonzero vector in $\R^M$ satisfying the equality $Du(x)=\lambda_1u(x)$ and therefore $\lambda_1\in\splus(D)$. The contradiction shows our assertion.

\noindent {\em Step 2}: If $C$ is an $M\times M$ square matrix, $\ker C\cap\R^m_+=\{0\}$, $v, w\in \R^M_+$, $Cv=\mu v$ with $\mu<0$ and $Cw=v$, then $C(\mu w)=\mu v$. Hence $v-\mu w\in\ker C$. Consequently $v=\mu w$. This implies that  $v=w=0$.

\noindent {\em Step 3}: Let $u=u^1 + u^\bot\in \ker(A-G)\cap K$, where $u^1\in E_1$, $u^\bot\in E_1^\bot$. The invariance of $E_1$ and $E_1^\bot$ under $A$ and $G$ yields $u^1\in \ker(A-G)$. From (2) above it follows that $u^1\in K$, therefore, by (3),  $u^1=0$. Finally, by (1) in Step 1, $u^\bot\in E_1^\bot\cap K=\{0\}$. This proves the first part of Proposition \ref{war-spektrum}.

\noindent {\em Step 4}: Let $\mu:=s(D)>\lambda_1$, let $0\neq v\in \R^M_+$ be an eigenvector of $D$ corresponding to $\mu$ and $\bar u=\phi v$. Then $\bar u\in K$. Suppose  $(A-G)u=\bar u$ for some $u=u^1+u^\bot\in K$.  The invariance of $E_1$ and $E_1^\bot$ under $A$ and $G$ shows that $(A-G)u^1=\bar u$, where $u^1\in K$ as in (2) above. Since $\l_1$ ia a simple eigenvalue, there is $w\in\R^M_+$ such that $u^1=\phi w$. Moreover $Au^1=\lambda_1\phi w$ and $G(u^1)=\phi Dw$. Since $\phi>0$ on $\Omega$, this implies that $(\lambda_1 I-D)w=v$.
Applying Step 2 to $C:=\lambda_1 I-D$ and $\l:= \l_1-\mu$ we see that $w=v=0$. This contradicts that $v\neq 0$.
\end{proof}

\begin{Th}\label{tw:wzordla_minusA+D_01} Let $D$, $G$ be as above and let $\lambda_1\not\in\splus(D)$. Then $Au=G(u)$, $u\in K$ iff $u=0$ and for any $r>0$
\begin{equation}
\deg_K(A,G,B_K(0,r))=\begin{cases}1&s(D)<\lambda_1\\0&s(D)>\lambda_1,\end{cases}
\end{equation}
where $B_K(0,r):=B(0,r)\cap K$.
\end{Th}
\begin{proof} The first part follows directly from Proposition \ref{war-spektrum}. Let $s(D)<\lambda_1$. Then $C=\{u\in K\cap D(A)\mid Au=tG(u)\; \text{for some}\; t\in [0,1]\}=\{0\}$ since if $u\in C$ and $u\neq 0$, then $\l_1\|u\|^2\leq \la Au,u\ra_E=t\la G(u),u\ra_E\leq t s(D)\|u\|^2$, which is impossible. Now $1=\deg_K(A,G;K)=\deg_K(A,G,B_K(0,r))$ in view of Proposition \ref{LS1} (b).\\
\indent Let now $s(D)>\lambda_1$. From Proposition \ref{war-spektrum} it follows that there is exists $\bar u\in K$ such that the $\{u\in K\cap D(A)\mid Au=G(u)+t\bar u,\; t\in [0,1]\}=\{0\}$. By the homotopy invariance $\deg_K(A,G;B_K(0,r))=\deg_K(A,G+\bar u,B_K(0,r))$. However there are no $u\in K\cap D(A)$ with $Au=G(u)+\bar u$. Hence   $\deg_K(A,G+\bar u,B_K(0,r))=0$.
\end{proof}
\noindent {\sc Linearization:} Now we generalize results from the previous section to the case of a nonlinear perturbation $\varphi$.
\begin{Ass}\label{linass}{\em In addition to assumption \ref{rd-system} we suppose that
\vspace{-4mm}
\begin{enumerate}
\item[($F_0$)] there is an $M\times M$ matrix $D_0$ such that $\vp(u)=D_0u+\vp_0(u)$ for $u\in R^M_+$, where $\vp_0(0)=\{0\}$ and
    \begin{equation}\label{zal 0}\lim_{u\to 0}\frac{\sup\{\|v\|\mid v\in\vp_0(u)\}}{|u|}=0.\end{equation}
\item[($F_\infty$)] there is  an $M\times M$ matrix $D_\infty$ such that $\vp(u)=D_\infty u+\vp_\infty(u)$ for $u\in R^M_+$, where
    $$\lim_{|u|\to \infty}\frac{\sup\{\|v\|\mid v\in \vp_\infty(u)\}}{|u|}=0.$$
\item[($R$)]  there are $M\times M$ matrices $R_0$ and $R_\infty$ such that $\rho(u)=R_0u+\rho_0(u)$, $\rho(u)=R_\infty u+\rho_\infty(u)$ for $u\in\R^M_+$, where
$$\lim_{u\to 0}\frac{\rho_0(u)}{|u|}=0,\ \ \lim_{|u|\to \infty}\frac{\rho_\infty(u)}{|u|}=0.$$
\end{enumerate}
}\end{Ass}

\begin{Rem}{\em (1) Observe that if $\varphi$ is single-valued then assumption ($F_0$) (resp. ($F_\infty$) means that $\varphi$ is differentiable at $0$ (resp. at infinity) with respect to the cone $\R^M_+$ and and $\varphi'(0)=D_0$.\\
\indent (2) Condition \eqref{fi-growth} together with ($F_0$) imply that there is $M>0$ such that
\begin{equation}\label{fi-growth 1}\sup_{v\in\varphi(u)}|v|\leq M|u|.\end{equation}
\indent (3) For all $u\in\R^m_+$ we have $D_0u,D_\infty u\in T_{\R^m_+}(u)$.
 Indeed, let $u\in\R^m_+$. Then $\varphi(tu)=tD_0u+\vp_0(tu)$ for all $t>0$.  Since $t^{-1}\vp(tu)\cap T_{\R^M_+}(u)\neq \emptyset$, passing with $t\to 0$ we obtain that $D_0u\in C_{\R^m_+}(u)$. The tangency of $D_\infty$ can be proved in a similar manner.\\
\indent (4) The growth condition  $|\rho(u)|\geq\alpha|u|$ shows that matrices $R_0$, $R_\infty$ are invertible and:
\begin{equation}\label{gamma}\lim_{u\to 0}\frac{\gamma(u)-R_0^{-1}u}{|u|}=0,\ \ \lim_{|u|\to \infty}\frac{\gamma(u)-R_\infty^{-1}u}{|u|}=0.\end{equation}
}\end{Rem}
Let $\Phi\colon K\multi E$ be the Nemytskii operator induced by $\vp$, i.e., $\Phi(u)=\{v\in E\mid v(x)\in\vp(u(x))\;\;\text{for a.a.}\; x\in\Omega\}$ for $u\in K$. Assumption \ref{linass} implies that $\Phi$ admits linearizations in the sense of Hadamard (\footnote{Simple examples show that the Nemytskii operator $N_\vp$ associated to $\vp$ (i.e., $N_\vp(u)(x)=\vp(u(x))$ for $u\in K$) admits neither the linearization at 0 nor at $\infty$, i.e, it is not necessarily true that
$\|u_n\|^{-1}\|w_n-N_{D_0}(u_n)\|\to 0$ if $u_n\to 0$ and $w_n\in N_\varphi(u_n)$; see \cite{Stuart}.}). Namely it is easy to see that
\begin{Prop}\label{Hadamard} For any sequences $(s_n)$, $s_n\to 0^+$ (resp. $s_n\to\infty$), $(u_n)$ in $K$, $u_n\to u\in K$ and $v_n\in \Phi(s_nu_n)$, one has that
$s_n^{-1}v_n\to G_0(u)$ (resp. $s_n^{-1}v_n\to G_\infty(u)$), where $G_0$ (resp. $G_\infty$) is the Nemytskii of $D_0$ (resp. $D_\infty$).
Consequently maps $H_0, H_\infty\colon K\times [0,1]\multi E$, given for $u\in K$ by
\begin{equation} \label{homot}
H_0(u,t):= \begin{cases}t^{-1}\Phi(tu),& t\in (0,1] \\
\big\{G_0(u)\big\}, & t=0, \end{cases}\quad\quad\quad
H_\infty(u,t):= \begin{cases}t\Phi(t^{-1}u),& t\in (0,1] \\
\big\{G_\infty(u)\big\}, & t=0, \end{cases}
\end{equation}
are $H$-usc with weakly compact convex values.
\end{Prop}
\begin{proof} Suppose $s_n\searrow 0$, $u_n\to u$ in $K$ and $v_n\in|phi(s_nu_n)$. We may assume that $u_n\to u$ a.e. and $|u_n|\leq \bar u\in L^2(\Omega)$ a.e. on $\Omega$. Let $z_n:=s_n^{-1}v_n$. By \eqref{fi-growth 1}, $|z_n|\leq M|u_n|\leq \bar u$ a.e. Take $x\in\Omega$ from the full measure set. If $u(x)\neq 0$, then for large $n\in\N$,
$$z_n(x)-D_0u_n(x)\in\big|u_n(x)\big|\frac{\vp_0\big(s_nu_n(x)\big)}{\big|s_nu_n(x)\big|}.$$
Thus $z_n(x)\to D_0u(x)$. If $u(x)=0$, then $|z_n(x)|\leq M|u_n(x)|\to 0$. Therefore the assertion follows in view of the Lebesgue dominated convergence theorem. The proof of the second part is analogous. \end{proof}
\begin{Rem}{\em In view of Assumption \ref{linass} (R) maps $\widetilde H_0,\widetilde H_\infty\colon K\times [0,1]\to E$
given for $u\in K$ by
\begin{equation*}
\widetilde H_0(u,t):= \begin{cases}t^{-1}\Gamma(tu),& t\in (0,1] \\
\big\{\Gamma_0(u)\big\}, & t=0, \end{cases}\quad\quad\quad \widetilde H_\infty(u,t):= \begin{cases}t\Gamma(t^{-1}u),& t\in (0,1] \\
\big\{\Gamma_\infty(u)\big\}, & t=0, \end{cases}
\end{equation*}
where $\Gamma$, $\Gamma_0$ and $\Gamma_\infty$ are Nemytskii operators induced by $\gamma$, $R_0^{-1}$ and $R_\infty^{-1}$, respectively, are continuous.
}\end{Rem}

\begin{Th}\label{tw:wzor-deg_diff0}  (i) Under Assumption \ref{rd-system} and \ref{linass} suppose that $\lambda_1\not\in\splus(D_0\circ{R_0^{-1}})$. Then there is $r>0$ such that $0$ is the unique solution to $Au\in F(u)$ in $B_K(0,\delta)$ and
$$\deg_K(A,F;B_K(0,r))=\begin{cases}1& \mathrm{if}\;\; s(D_\infty\circ{R_\infty^{-1}})<\lambda_1 \\
0&\mathrm{if}\;\;  s(D_\infty\circ{R_\infty^{-1}})>\lambda_1. \end{cases}$$
\indent (ii) If $\lambda_1\not\in\splus(D_\infty\circ{R_\infty^{-1}})$, then there is $R>0$ such if $u\in K\cap D(A)$ and $A(u)\in F(u)$, then $\|u\|<R$ and
$$\deg_K(A,F;B_K(0,R))=\begin{cases}1& \mathrm{if}\;\; s(D_\infty\circ{R_\infty^{-1}})<\lambda_1 \\
0&\mathrm{if}\;\;  s(D_\infty\circ{R_\infty^{-1}})>\lambda_1. \end{cases}$$
\end{Th}
\begin{proof}
(i) Define $H\colon K\times [0,1]\multi E$ by $H(u,t):=H_0\circ \widetilde H_0(u,t)$, $u\in K$, $t\in [0,1]$, where $H_0$ and $\Gamma_0$ are given by \eqref{homot}. It is clear that $H$ is well defined, $H$-usc with convex weakly compact values. We shall show that there is $r>0$ such that $0$ is the unique solution of $Au\in H(u,t)$ in $B_K(0,r)$.\\
\indent Suppose to the contrary that there is a sequence $(u_n)$ in $K$ such that $u_n\neq 0$ for all $n\in\N$ and $v_n:=A u_n\in F(u_n)$. From \eqref{fi-growth 1} we see that the sequence $\|u_n\|^{-1}v_n$ is bounded. Therefore the sequence $(w_n)$, where $w_n:=J_1^A\big(\|u_n\|^{-1}(u_n+v_n)\big)=\|u_n\|^{-1}u_n$, has the convergent subsequence, wlog we may assume that  $w_n\to w$. \\
\indent On the other hand $v_n\in F\big(\|u_n\|w_n\big)$. Proposition \ref{Hadamard} yields that $\|u_n\|^{-1}v_n\to G_0\circ \Gamma_0(w)$ and, consequently, $w= J_1^A(w+G_0\circ\Gamma_0(w))$, i.e., $A w= G_0\circ\Gamma_0(w)$. Thus $w=0$ in view of the first part of Proposition \ref{war-spektrum}. This a contradiction with $\|w\|=1$.\\
\indent The homotopy invariance implies that $\deg_K\big(A,F;B_K(0,r)\big)=\deg_K\big(A,G_0\circ\Gamma_0,B_K(0,r)\big)$. To complete the proof it remains  to apply Theorem \ref{tw:wzordla_minusA+D_01}.\\
\indent (ii) The proof is similar to the above on; one considers the homotopy $H_\infty\circ \widetilde H_\infty$ and shows that there is $R>0$ such that all solution to problem $Au\in H_\infty\circ \widetilde H_\infty(u,t)$, $u\in K$, are contained in a large ball $B(0,R)$.
\end{proof}
Now we are prepared to establish the main theorem of this section.
\begin{Th}\label{tw:istnienie-inkluzja-rozn}  Under Assumption \ref{rd-system} and \ref{linass} suppose that
\begin{align*}\lambda_1\not\in\splus\left(D_0R_0^{-1}\right)\cup\splus\left(D_\infty R_\infty^{-1}\right)\;\; \text{and}\;\;
\Big(s\big(D_0R_0^{-1}\big)-\lambda_1\Big)\cdot \Big(s\big(D_\infty R_\infty^{-1}\big)-\lambda_1\Big)<0,\end{align*} then there exists at least one nontrival solution to Problem \eqref{eq:Abstr Incl intro} and \eqref{reac-diff1}.
\end{Th}
\begin{proof} We see that there are $0<r<R$ such that $\deg_K\big(A,F;B(0,r)\big)\neq\deg_K\big(A,F;B(0,R)\big)$. The additivity property of the degree implies that there is $u_0\in K$ such that $Au_0\in F(u_0)$ and $r<\|u_0\|<R$.\end{proof}
\begin{Rem}{\em Instead of condition (2) from Assumption \ref{rd-system} assume that, for $u\in \R^M_+$, if $v\in\vp(u)$ and $u_i=0$, then $v_i\geq 0$, $i=1,\ldots,M$. Then $\vp$ is {\em strongly} tangent, i.e. $\vp(u)\subset T_{\R^M_+}(u)$ for any $u\in\R^M_+$. Moreover, in this situation, conditions $(F_0)$ and $(F_\infty)$ from assumption \ref{linass} may be slightly relaxed. Namely instead of $(F_0)$ and $(F_\infty)$ suppose that there are two $M\times M$ quasi-nonnegative matrices $D_0,D_\infty$ such that
\begin{equation}\label{eq:tymczas2}\lim_{u\to 0}\frac{\dist\big(D_0u,\varphi(u)\big)}{|u|}=0,\;\;  \lim_{|u|\to \infty}\frac{\dist\big(D_\infty u,\varphi(u)\big)}{|u|}=0.\end{equation}
Then the conclusion of Theorem \ref{tw:istnienie-inkluzja-rozn} remains true.\\
\indent Indeed, in view of (\ref{eq:tymczas2}) there are a continuous function $\alpha\colon [0,\infty)\to[0,\infty)$ and numbers $\underline t<\overline t$ such that
\begin{gather*}\lim_{t\to 0}\frac{\alpha(t)}{t}=\lim_{t\to\infty}\frac{\alpha(t)}{t}=0,\quad \dist\big(D_0u,\varphi(u)\big)<\alpha\big(|u|\big)\textrm{ if }|u|<\underline t\quad\text{and}\quad \dist\big(D_\infty u,\varphi(u)\big)<\alpha\big(|u|\big)\textrm{ if }|u|>\overline t.\end{gather*}
The map $\psi\colon K\multi\R^M$ given for $u\in\R^M_+$  by \[\psi(u)=\begin{cases}\varphi(u)\cap \Big(D_0(u)+\alpha\big (|u|\big)\ov B\Big)& |u|<\underline t\\
\varphi(u)&\underline t\leq|u|\leq\overline t\\
\varphi(u)\cap \Big(D_\infty(u)+\alpha\big(|u|\big) \ov B\Big)& |u|>\overline t
\end{cases}\]
is upper semicontinuous with nonempty convex compact values and satisfies  conditions $(F_0)$ and $(F_\infty)$ with $\vp$ replaced by $\psi$. Therefore problem \eqref{reac-diff1} with $\vp$ replaced by $\psi$ has a nontrivial solution $u_0\in K$. Since for all $u\in K$ we have $\psi(u)\subset\vp(u)$, $u_0$ is the solution to the original problem.
}\end{Rem}

\section{Reaction-diffusion equations with discontinuities} \label{sect:reac-diff-nieciaglosc}
Results of the above section admit to study  the existence of solutions of reaction-diffusion systems with a discontinuous right-hand side.\\
\indent Consider a {\em possibly discontinuous} function $f\colon \R^M_+\to\R^M$  and the following reaction-diffusion system
\begin{equation} \label{eq:udc-uklad-rownan-f}
  \left\{ \begin{array}{ll} -(\Delta (\rho\circ u))(x)=f(u(x)) \\ u_i(x)\geq 0,\;x\in\Omega,\ i=1...m \\  u_{|\part\Omega}=0.\end{array}\right.
\end{equation}
In order t
\[apply the above approach one needs \text{he n }oion of a regularization allo\]ing to remove discontinuities. With  $f$ the so-called {\em Krasowski regularization}
\[r_K(f)(x): =\bigcap_{\varepsilon>0}\cl{\mathrm{conv}}f(B^+(x,\varepsilon)),\;\; x\in\R^M_+,\]
where $B^+(x,\eps)=\{y\in \R^M_+\mid |x-y|<\eps\}$, and the {\em Fillipov regularization}
\[r_F(f)(x): =\bigcap_{\e>0}\bigcap_{\mu(N)=0}\cl{\mathrm{conv}}f(B^+(x,\e)\setminus N),\] where $\mu$ stand for the Lebesgue measure, are associated. Obviously, 
\[\{f(x)\},r_F(f)(x)\subset r_K(f)(x)\text{ and } r_F(f)(x)=r_K(f)(x)=\{f(x)\}\] if $f$ is continuous in $x$.  It is immediate to show that $r_K(f)$ and $r_F(f)$ are usc with compact convex values.\\
\indent When solving \eqref{eq:udc-uklad-rownan-f} the common practise is to consider solutions in terms of regularization. Namely, a nonnegative function $u$ is  a \emph{solution} of \eqref{eq:udc-uklad-rownan-f} if it is the solution of the inclusion \eqref{reac-diff1}, where $\vp$ is a regularization of $f$ (comp. e.g. \cite{KCC, Fil, Bothe1, Carl, CC}).

\indent Let conditions (4) from Assumption \ref{rd-system}, condition ($R$) from Assumption \ref{linass} be satisfied and
\begin{enumerate}
\item [(a)] there is $c>0$ such that $|f(u)|\leq c(1+|u|)$ for $u\in\R^m_+$;
\item [(b)] $f(u)\in T_{\R^m_+}(u)$ dla $u\in\R^m_+$.
\item [($f$)] $f(u)=D_0u+\beta_0(u)=D_\infty u+\beta_\infty(u)$, where $D_0, D_\infty$ are $M\times M$ quasi-nonnegative matrices and
    $$\beta_0(0)=0,\ \ \lim_{u\to 0}\frac{|\beta_0(u)|}{|u|}=0,\quad \lim_{|u|\to \infty}\frac{|\beta_\infty(u)|}{|u|}=0.$$
\end{enumerate}

In view ($f$) and the definition of regularizations we see that $\varphi=r_K(f)$ or $\vp=r_F(f)$ satisfies the conditions ($F_0$) and ($F_\infty$) from Assumption \ref{linass}. In case of Krasowski regularization condition (b) above implies the tangency of $\varphi$ to $\R^m_+$, i.e. condition (2) in Assumption \ref{rd-system}. In case of Filippov regularization the same implication may not to be true. To overcome this difficulty one can assume additionally the continuity of $f$ at points from the boundary of $\R^m_+$. \\
\indent Assuming that the regularization $\varphi$ is tangent to $\R^m_+$ we get the following result.
\begin{Th}\label{tw:istnienie-nieciaglosc} Assume that $\lambda_1\not\in\splus(D_0R_0^{-1})\cup\splus(D_\infty R_\infty^{-1})$.
If $(s(D_0R_0^{-1})-\lambda_1)\cdot(s(D_\infty R_\infty^{-1})-\lambda_1)<0$,
then there exists at least one nontrivial solution of the problem (\ref{eq:udc-uklad-rownan-f}).\hfill $\square$
\end{Th}
Above solutions to \eqref{eq:udc-uklad-rownan-f} were understood in the sense of the Krasowski or Filippov regularizations. The natural question is whether u such that  $-\Delta(\rho\circ u)\in\varphi(u)$ satisfies $-\Delta(\rho\circ u)=f(u)$ almost everywhere. Such solutions will be called {\em primitive}. It is not difficult to find examples showing that in general solutions are not primitive. However we have the following
\begin{Th}\label{tw:pierwotne_rozwiazania} If the set of discontinuities of $f$ is at most countable and
\begin{equation}\label{eq:tmp17}0\not\in\varphi(u)\setminus\{f(u)\}\end{equation} for all $u\in\R^m_+$, then any solution $u$ of (\ref{eq:udc-uklad-rownan-f}) in the sense of regularization $\varphi=r_K(f)$ or $r_F(f)$ is a primitive solution.
\end{Th}
\begin{proof} The assumption implies that $\varphi(u)=\{f(u)\}$ fails to hold for at most countably many $u\in \R^m_+$. Let $u\colon\Omega\to\R^m_+$ be a solution. It is known (see \cite{Evans}) that $u\in D(\Delta)=H^2(\Omega,\R^m)\cap H^1_0(\Omega,\R^m)$. Let $u_0\in\R^M_+$ be a fixed point satisfying the relation $\varphi(u_0)\neq\{f(u_0)\}$. Put $\Omega_{u_0}:=\{x\in\Omega\mid u(x)=u_0\}$. The function $u$ is constant on $\Omega_{u_0}$, therefore (\cite[Lemma 7.7]{GT})
$\nabla u(x)=0$ for a.a. $x\in\Omega_{u_0}$. The same argument applied to $\nabla u\in H^1(\Omega,\R^{m\times m})$ gives that
$-\Delta u(x)=0$ for a.a. $x\in\Omega_{u_0}$. On the other hand
$0=-\Delta u(x)\in \varphi(u(x))=\varphi(u_0)$ for a.a. $x\in\Omega_{u_0}$.
The assumption (\ref{eq:tmp17}) implies that the set $\Omega_{u_0}$ is of measure zero or $f(u_0)=0$. In both cases the equation \eqref{eq:udc-uklad-rownan-f} is satisfied for almost all $x\in\Omega_{u_0}$.\\
\indent Summarising, the equation \eqref{eq:udc-uklad-rownan-f} is satisfied for almost all $x\in\Omega$ for which $\varphi(u(x))\neq\{f(u(x))\}$. For the remaining points the equation is trivially satisfied.\end{proof}

\noindent \textbf{Acknowledgement:} The first author was supported by the internal grant of the Lodz University of Technology.


\end{document}